\documentclass[11pt,a4paper,reqno]{article}

\usepackage[colorlinks,pdfstartview=FitH,citecolor=purple,linkcolor=purple,urlcolor=purple]{hyperref}

\usepackage{cite}
\usepackage{tikz}\usetikzlibrary{matrix,decorations.pathreplacing,positioning}
\usepackage{hyperref}
\usepackage{amsmath,amsthm,amssymb,bm}
\usepackage{rotating}
\usepackage{setspace}
\usepackage{algorithm}
\usepackage{algpseudocode}

\usepackage{multirow}
\usepackage{array}
\newcolumntype{x}[1]{>{\centering\arraybackslash\hspace{0pt}}p{#1}}
\usepackage{empheq}
\usepackage{wasysym}
\usepackage{subcaption}
\usepackage{subcaption}

\usepackage{enumitem}
\setitemize{itemsep=-1pt}
\setenumerate{itemsep=-1pt}

\usepackage[margin=2.4cm]{geometry}
\usepackage{pgfplots}
\pgfplotsset{width=10cm,compat=1.9}

\definecolor{myg}{RGB}{220,220,220}
\usepackage{etoolbox}% http://ctan.org/pkg/etoolbox
\AtEndEnvironment{example}{\null\hfill$\blacktriangle$}%
\usepackage{caption}

\theoremstyle{definition}
\newtheorem{theorem}{Theorem}[section]
\newtheorem{corollary}[theorem]{Corollary}
\newtheorem{proposition}[theorem]{Proposition}
\newtheorem{lemma}[theorem]{Lemma}

\newtheorem{definition}[theorem]{Definition}
\newtheorem{example}[theorem]{Example}
\newtheorem{notation}[theorem]{Notation}
\newtheorem{remark}[theorem]{Remark}

\newcommand*{\myproofname}{Proof of the claim}

\newcommand{\numberset}{\mathbb}
\newcommand{\N}{\numberset{N}}
\newcommand{\Z}{\numberset{Z}}
\newcommand{\Q}{\numberset{Q}}
\newcommand{\R}{\numberset{R}}

\newcommand{\F}{\numberset{F}}

\newcommand{\mC}{\mathcal{C}}
\newcommand{\mP}{\mathcal{P}}

\newcommand{\mA}{\mathcal{A}}

\newcommand{\wt}{\textnormal{wt}}
\newcommand{\wH}{\textnormal{wt}^\textnormal{H}}
\newcommand{\PG}{\textnormal{PG}}

\newcommand{\dH}{d^\textnormal{H}}

\newcommand{\vol}{\textnormal{vol}}

\newcommand{\mR}{\mathcal{R}}

\newcommand{\all}{\textnormal{all}}

\newcommand{\conv}{\mathrm{conv}}
\newcommand{\rank}{\textnormal{rank}}

\newcommand{\floor}[1]{{\left\lfloor{#1}\right\rfloor}}

\newlength{\mynodespace}
\renewcommand{\vert}{\textnormal{vert}}

\let\OLDthebibliography\thebibliography
\renewcommand\thebibliography[1]{
  \OLDthebibliography{#1}
  \setlength{\parskip}{0pt}
  \setlength{\itemsep}{6pt plus 0.3ex}
}

\title{\textbf{The Service Rate Region Polytope}}

\usepackage{authblk}

\author[1]{G. N. Alfarano\thanks{G. N. A. is supported by the Swiss National Foundation through grant no. 210966.}}

\author[2]{A. B. K\i l\i\c{c}\thanks{A. B. K. is supported by the Dutch Research Council through grant VI.Vidi.203.045.}}
\author[3]{A. Ravagnani\thanks{A. R. is supported by the Dutch Research Council through grants VI.Vidi.203.045, 
OCENW.KLEIN.539, 
and by the Royal Academy of Arts and Sciences of the Netherlands.}}

\author[4]{E. Soljanin}

\affil[1,2,3]{Eindhoven University of Technology, the Netherlands}
\affil[4]{Rutgers University, NJ, U.S.A.}

\date{}

\begin{document}

\thispagestyle{empty}
\maketitle

\begin{abstract}
We investigate the properties of a family of polytopes that naturally arise in connection with a problem in distributed data storage, namely service rate region polytopes. The service rate region of a distributed coded system describes the data access requests that the underlying system can support. In this paper, we study the polytope structure of the service rate region with the primary goal of describing its geometric shape and properties.
We achieve so by introducing various structural parameters of the service rate region and establishing upper and lower bounds for them. The techniques we apply in this paper range from coding theory to optimization. One of our main results shows that every rational point of the service rate region has a so-called rational allocation, answering an open question in the research area. 
\end{abstract}

\medskip

\section*{Introduction}

Distributed storage systems split data across servers to provide access services to multiple, possibly concurrent users. 
The simplest way to reliably handle concurrent requests is to replicate data according to their popularity; see for instance \cite{patterson1988case,shvachko2010hadoop}. Unfortunately, this method can be expensive in terms of storage. Moreover, predicting how the interest in data changes is not always easy. For these reasons, erasure-coding has gained attention as a form of redundant storage; see e.g. \cite{dimakis2011survey} and references therein. 

Recent work establishes the concept of \emph{service rate region} as an essential measure of the efficiency of a distributed coded system that should be considered in the system's design phase; see  \cite{service:aktas2021jkks, service:KazemiKSS21g, service:KazemiKSS20c, service:KazemiKSS20g, service:AndersonJJ18, service:AktasAJ17}. To understand this metric, consider distributed systems in which $k$ data objects are stored, using a linear $[n, k]_q$ error-correcting code across $n$ servers, each with the same capacity $\mu\in\R$. The \emph{service rate region} of the distributed coded system is the set of all request rates $(\lambda_1,\ldots,\lambda_k) \in \R^k$ that the system can simultaneously handle. Such a distributed system is defined by a full rank $k\times n$ matrix $G$, and its service rate region is a convex polytope in~$\R^k$.

Previous work in the area focuses on characterizing the service rate region of a distributed coded system and finding the optimal strategy to split the rate requests across the servers to maximize the region; see, e.g., \cite{service:aktas2021jkks}.
The service rate region has been characterized for \textit{binary simplex codes} and two classes of \textit{maximum distance separable} (\textit{MDS}) \textit{codes}: 1) \textit{systematic} MDS codes when $n\geq 2k$ and 2) MDS codes with arbitrary length and dimension that do not permit any data objects decoding from fewer than $k$ stored objects; see again \cite{service:aktas2021jkks} and references therein. 

A combinatorial approach to the service rate region has been introduced in \cite{service:KazemiKSS20c}, establishing and using the equivalence between the service rate problem and the well-known fractional matching problem on hypergraphs. In the same work, the authors showed that the service rate problem generalizes, in a precise sense, \textit{batch}, \textit{PIR}, and \textit{switch codes}; see \cite{ishai2004batch,riet2018asynchronous,skachek2018batch,fazeli2015pir, fazeli2015codes, tamo2014family} for the details and background material about these classes of codes. In \cite{service:KazemiKSS20g}, the service rate regions of the \textit{binary first-order Reed-Muller codes} and \textit{binary simplex codes} have been determined using a geometric approach. In \cite{alfarano2022dual}, coding-theoretic tools have been used to identify a polytope that contains the service rate region, giving an outer bound for it.

\bigskip

\noindent\textbf{Contributions.} \, 
In contrast with previous approaches, this paper focuses on describing the polytope structure of the service rate region and its geometric properties, such as its volume.
Key tools to achieve this goal are outer bounds for the service rate region polytope, which we obtain by applying methods ranging from coding theory to convex geometry.

We also propose a discretized notion
of service rate region that
arises naturally in the integer allocation model. We prove that the discretized notion returns precisely the rational points of the 
originally proposed (continuous) notion.
When investigating the connection between the allocation and the service rate region polytopes, we show that every rational point of the region has a rational allocation.
The last part of the paper focuses on the service rate regions of systematic MDS matrices, for which we compute, for example, the volume in dimensions 2 and 3.

The rest of the paper is organized as follows: Section \ref{sec:1} introduces access models and states the service rate problem. Section \ref{sec:2} is devoted to the various representations of the service rate region polytope. In Section \ref{sec:3}, we discretize the concept of service rate region, also showing that rational points in the region have rational allocation. In section \ref{sec:4}, we introduce and study the $r$-th max-sum capacity and the system's volume. Section \ref{sec:5} is devoted to proving different outer bounds on the service rate region. Section \ref{sec:6}  focuses on systematic MDS codes. 
Elementary background 
on polytopes and error-correcting codes
is provided in the appendices.

\section{System Model and Problem Formulation}
\label{sec:1}

We consider a distributed service system with $n$ identical nodes (servers). Each node has two functional components: one for data storage and the other for processing download requests posed by the system users.
This section establishes the notation, defines distributed service systems and their service rate regions, and states the problems this paper addresses.

Throughout the paper, $q$ denotes a prime power and 
$\F_q$ is the finite field with $q$ elements.
We work with integers $n > k \ge 2$. All vectors in the sequel are row vectors unless otherwise stated.

\subsection{Storage Model}
\label{subs:storage_model}
We consider a coded, distributed data storage system where~$k$ objects
(elements of $\F_q$) are linearly encoded and stored across $n$ servers. Each server stores precisely one element of $\F_q$. 
Therefore, the coded system is specified by a rank~$k$ matrix $G \in \F_q^{k \times n}$, which we call the \textbf{generator matrix} of the system. If $(x_1,\dots,x_k) \in \F_q^k$ is the $k$-tuple of objects to be stored,
then the $j$-th server stores the $j$-th component of the encoded vector
$$(x_1,\dots,x_k) \cdot G \in \F_q^n.$$
Note that, by definition, the $n$ servers store \textit{linear combinations} of the data objects rather than just copies of them. The latter storage strategy is called (\textit{simple}) \textit{replication}.

Following the coding theory terminology~\cite{macwilliams1977theory}, we say that the matrix $G$ is \textbf{systematic} if 
its first~$k$ columns form the identity $k \times k$ matrix. If the $\nu$-th column of $G$ is a nonzero multiple of the standard basis vector $e_i$, then we say that $\nu$ is a \textbf{systematic node} for the $i$-th data object. If every $\nu \in \{1, \ldots, n\}$
is a systematic node, then we say that
$G$ is a \textbf{replication matrix}.\label{pagerepl}
Note that a replication matrix describes a system where each object is stored as it is (up to nonzero multiples).

\begin{notation} \label{not:whatG}
To simplify the statements throughout the paper, without loss of generality,
we work with a fixed matrix $G \in \F_q^{k \times n}$ of rank $k$. We denote by $G^\nu$ the $\nu$-th column of $G$ and assume that none of its columns is the zero vector. 
\end{notation}

Since $n>k$, we may be able to recover
each object from different sets of
servers, which motivates the following definitions and terminology.

\begin{definition} \label{maindef}
Let $R\subseteq \{1,\dots,n\}$ be such that $e_i \in \langle G^\nu \mid \nu \in R\rangle$ and $\langle G^\nu \mid \nu \in R\rangle$ is the $\F_q$-span of the columns of~$G$ indexed by $R$. Then $R$ is called a \textbf{recovery set} for the $i$-th object.
For $i \in \{1,\dots,k\}$, let
$$\mR^\all_i(G):=\{R \subseteq \{1,\dots,n\} \mid e_i \in \langle G^\nu \mid \nu \in R\rangle\},$$
where superscript
``all'' indicates that $\mR_i^\all(G)$ contains \textit{all} the recovery sets for the $i$-th object. 
\end{definition}
Since $G$ has rank $k$ by assumption, we have 
$\mR^\all_i(G) \neq \emptyset$ for all $i \in \{1,\ldots,k\}$. Moreover, $R \neq \emptyset$ for all $i\in\{1,\dots,k\}$ and $R \in \mR^\all_i(G)$. We continue by formalizing the concept of a recovery system.

\begin{definition}
A (\textbf{recovery}) $G$-\textbf{system} is a $k$-tuple
$\mR=(\mR_1,\ldots,\mR_k)$,
with 
$\mR_i \subseteq \mR_i^\all(G)$ and $\mR_i \neq \emptyset$
for all $i \in \{1, \ldots,k\}$.
\end{definition}

\subsection{Service and Access Models}
We adopt two service models and refer to them as the \textit{queuing} and the \textit{bandwidth} model. They are based on two common ways of implementing resource sharing by incoming data access requests. 
In the queueing model, the requests to download from a node are placed in its queue. Each node can serve on average $\mu$ requests per unit of time. To maintain the stability of each queue, the total request arrival rate at each node must not exceed its service rate $\mu$. 

In the bandwidth model, each node can concurrently serve multiple data access requests. When each node has an I/O bus with a finite access bandwidth $W$ bits/second and a download requires streaming at a fixed bandwidth of $b$ bits/second, a node can simultaneously serve up to $\mu=\floor{W/b}$ number of requests. In both cases, we refer to $\mu$ as the \textbf{server's capacity} (formal definitions will be given later).
In the queuing model, requests to download object $i$ arrive at rate $\lambda_i \in \R_{\ge 0}$. In the bandwidth model, $\lambda_i$ is the number of object $i$ requests simultaneously in the system. 
In both models, $\lambda_{i,R}$ is the portion of $\lambda_i$ assigned to be served by the recovery set $R \in \mR_i$.
We refer to a set $\{\lambda_{i,R} \mid R\in \mathcal{R}_i$, $i=1,\dots,k\}$ as a \textbf{request allocation}.

\subsection{Normalization and Integrality} 
\label{subs:normal}
 We can normalize all request allocation values and rates by dividing them by the node service capacity $\mu$. 
 In this case, all normalized request rates $\lambda_i$ and the numbers in $\{\lambda_{i,R} \mid R\in \mathcal{R}_i$, $i \in \{1,\ldots,k\}\}$ are multiples of $1/\mu$. 
 Since in the bandwidth model, these numbers count requests, their normalized versions are integer multiples of $1/\mu$. 
 Furthermore, there are practical scenarios wherein each served request occupies the entire bandwidth of the server they are accessing (e.g., streaming from low-bandwidth edge devices). In such cases, $\lambda_i$ are integers, and $\lambda_{i,R}$ are binary numbers. In other practical scenarios, a user can simultaneously download data from multiple nodes at a fraction of its bandwidth from each, and the assumption that $\lambda_{i,R}$ be integer multiples of $1/\mu$ can be relaxed.

\subsection{Service Rate Region and Problem Formulation}
We are interested in characterizing the $k$-tuples $(\lambda_1,\ldots,\lambda_k) \in \R^k$ of rate requests that the data storage system can support. 
The set of such tuples is formally defined as follows, yielding to the notion of 
the \textit{service rate region} of a distributed storage system.

\begin{definition} \label{def:system}
Let $\mR=(\mR_1,\ldots,\mR_k)$ be
a $G$-system.
The \textbf{service rate region} associated with~$\mR$ and~$\mu$ is the set of
all $(\lambda_1,\ldots,\lambda_k) \in \R^k$ for which there exists a collection of real numbers $$\{\lambda_{i,R} \mid i \in \{1,\ldots,k\}, \, R \in \mR_i\}$$
with the following properties:
 \begin{empheq}[left=\empheqlbrace]{align}
    \sum_{R \in \mR_i}\lambda_{i,R} = \lambda_i & \textnormal{ for } 1\leq i\leq k, \label{c1}\\
    \sum_{i=1}^k \sum_{\substack{R \in \mR_i \\ 
  \nu \in R}} 
    \lambda_{i,R} \leq \mu &\textnormal{ for } 1\leq \nu \leq n, \label{c2} \\
    \lambda_{i,R} \geq 0  &\textnormal{ for } 1\leq i\leq k, ~ R \in \mR_i. \label{c3}
\end{empheq}

A collection $\{\lambda_{i,R}\}$ that satisfies  properties 
\eqref{c2} and \eqref{c3} above is called a \textbf{feasible allocation} for the pair $(\mR,\mu)$. The service rate region associated with $\mR$ and $\mu$ is denoted by $$\Lambda(\mR,\mu) \subseteq \R^k.$$ \end{definition}

Observe that the service rate region of a $G$-system is independent of the ordering of the recovery sets in each $\mR_i$.

\begin{remark}\label{rem:down-monotone}
It turns out that
$\Lambda(\mR,\mu)$ is a down-monotone polytope; see Theorem~\ref{prop:def_f} below and Appendix~\ref{sec:appendix1} for the definitions.
We will elaborate on this when introducing the allocation polytope in Section~\ref{sec:2}.
\end{remark}

The service rate region of a $G$-system $\mR$ may not change if we select a suitable subset of the recovery sets, which allows us to reduce the number of variables and inequalities in Definition~\ref{def:system}.
We start with the following observation, whose proof is simple and therefore omitted.

\begin{proposition}\label{prop:subsystem}
Suppose that $\mR=(\mR_1,\ldots,\mR_k)$ and 
$\mR'=(\mR'_1,\ldots,\mR'_k)$ are $G$-systems with $\mR'_i \subseteq \mR_i$ for all 
$i \in \{1,\ldots,k\}$. Then
$\Lambda(\mR,\mu) \supseteq \Lambda(\mR',\mu)$. In particular,
$\Lambda(\mR,\mu) \subseteq 
    \Lambda(\mR^{\all}(G),\mu)$
for any $G$-system $\mR$.
\end{proposition}

The service rate region $\Lambda(\mR^{\all}(G),\mu)$ does not change when we select from $\mR^{\textnormal{all}}(G)$ the recovery sets that are 
minimal with respect to inclusion, in the following precise sense.

\begin{definition}
A set $R \in \mR^\all_i$ is called 
\textbf{$i$-minimal} if there is no $R' \in \mR^\all_i(G)$ with $R' \subsetneq R$. We let $\mR^{\textnormal{min}}(G)$ be the $G$-system defined, for all $i$, by
$$\mR_i^{\textnormal{min}}(G):=\{R \in \mR^\all_i(G) \mid R \mbox{ is $i$-minimal}\}.$$
\end{definition}

The proof of the following result is not difficult and therefore left to the reader.

\begin{proposition}
\label{lostesso}
We have
$\Lambda(\mR^{\textnormal{min}}(G),\mu) = \Lambda(\mR^\all(G),\mu)$.
\end{proposition}

\begin{remark}\label{rem:scale_mu}
It immediately follows from the definitions that $\Lambda(\mR,\mu) = \mu \Lambda(\mR,1)$
for any $G$-system~$\mR$, where
$\mu \Lambda(\mR,1)=\{\mu \lambda \mid \lambda \in \Lambda(\mR,1)\}$.  In the sequel, we will often assume $\mu=1$ without loss of generality.
\end{remark}

The following symbols will further simplify the statements in the sequel.

\begin{notation}
For a $G$-system $\mR=(\mR_1,\ldots,\mR_k)$ 
we let $\Lambda(\mR)=\Lambda(\mR,1)$. We also
write $\Lambda(G,\mu)=\Lambda(\mR^\all(G),\mu)=
\Lambda(\mR^{\min}(G),\mu)$, where the latter equality follows from Proposition~\ref{lostesso}.
Finally, we set $\Lambda(G)=\Lambda(G,1)$.
\end{notation}

We conclude with an example illustrating the concepts introduced in this section.

\begin{example}\label{ex:different_examples} 
Consider the matrices 
\begin{align*}
G_1 = \begin{pmatrix}
1 & 0 & 1 & 1  \\
0 & 1 & 0 & 0 
\end{pmatrix}\in\F_2^{2\times 4}, \qquad 
G_2 = \begin{pmatrix}
1 & 0 & 0 & 1 & 0 & 1  \\
0 & 1 & 0 & 1 & 2 & 2 \\
0 & 0 & 1 & 1 & 1 & 1
\end{pmatrix}\in\F_3^{3\times 6}.
\end{align*}
The corresponding service rate regions are
depicted in Figure~\ref{fig:example1}.
For the matrix $G_2$ we have 
\begin{align*}
    \mR_1^{\min}(G_2)&=\{\{1\},\{5,6\},\{2,3,4\},\{2,4,5\}, \{3,4,6\},\{2,3,6\},\{3,4,5\}\},\\
    \mR_2^{\min}(G_2)&=\{\{2\},\{3,5\},\{4,6\},\{1,3,4\},\{1,4,5\},\{1,3,6\}\},\\
    \mR_3^{\min}(G_2)&=\{\{3\},\{2,5\},\{1,2,4\},\{1,4,6\},\{1,2,6\},\{1,4,5\}\}.
\end{align*}
Moreover, for all $i\in\{1,2,3\}$ we have
$\mR_i^{\all}=\{R\subseteq \{1,\ldots,6\} \mid S\subseteq R, \, \textnormal{for some } S\in\mR_i^{\min}(G_2)\}$.
Finally, to see that, for example,
the point $P=(3/2,3/2,1/2)$ belongs to $\Lambda(G_2)$, we can consider the feasible allocation 
given by
\begin{align*}
    \lambda_{1,R}=\begin{cases}
        1 & \textnormal{ if } R=\{1\},\\
        1/2 & \textnormal{ if } R=\{5,6\},\\
        0 & \textnormal{ otherwise,}
    \end{cases}  \qquad 
     \lambda_{2,R}=\begin{cases}
        1 & \textnormal{ if } R=\{2\},\\
        1/2 & \textnormal{ if } R=\{3,5\},\\
        0 & \textnormal{ otherwise,}
    \end{cases}  \qquad
     \lambda_{3,R}=\begin{cases}
        1/2 & \textnormal{ if } R=\{3\},\\
        0 & \textnormal{ otherwise.}
    \end{cases}
\end{align*}
It is easy to see that the collection $\{\lambda_{i,R}\}$ satisfies 
the properties \eqref{c1}--\eqref{c3}
for $(\mR^{\min}(G_2),1)$.
\end{example}

\vspace{-0.5cm}

\begin{figure}[hbt]
\centering
\subcaptionbox{$\Lambda(G_1)$,  Example~\ref{ex:different_examples}.\label{subfig:1}}%
  [.4\linewidth]{\begin{tikzpicture}[thick,scale=0.7]
\coordinate (A1) at (0,0);
\coordinate (A2) at (0,1);
\coordinate (A3) at (3,0);
\coordinate (A4) at (3,1);
\node at (A2) [left = 1mm of A2] {1};
\node at (A3) [below = 1mm of A3] {3};
\draw[fill=blue!50,opacity=0.4] (A1) -- (A2) -- (A4) -- (A3) -- cycle;
\draw[->, thick,black] (0,0)--(3.5,0) node[right]{$\lambda_1$};
\draw[->, thick,black] (0,0)--(0,1.5) node[above]{$\lambda_2$};
\draw[thick,black] (0,1)--(3,1)--(3,0);
\end{tikzpicture}}
\subcaptionbox{$\Lambda(G_2)$,  Example~\ref{ex:different_examples}.\label{subfig:2}}
  [.4\linewidth]{\begin{tikzpicture}[thick,scale=0.9]
\coordinate (A1) at (0,0,0);
\coordinate (A2) at (0,3,0);
\coordinate (A3) at (0,0,3);
\coordinate (A4) at (3,0,0);
\coordinate (A5) at (3/2,1/2,2);
\coordinate (A6) at (1,1,2);
\coordinate (A7) at (1,0,5/2);
\coordinate (B1) at (3,0,1);
\coordinate (B2) at (1,2,1);

\begin{scope}[thick,dashed,opacity=0.5]
\draw (A1) -- (A2);
\draw (A1) -- (A3);
\draw (A1) -- (A4);

\end{scope}

\draw[fill=blue!50,opacity=0.4] (A2) -- (B2) -- (A6) -- (A3) -- cycle;
\draw[fill=blue!50,opacity=0.4] (A3) -- (A6) -- (A5) -- (A7) -- cycle;
\draw[fill=blue!50,opacity=0.4] (A5) -- (A7) -- (B1) -- cycle;
\draw[fill=blue!50,opacity=0.4] (A5) -- (A6) -- (B2) -- (B1) -- cycle;
\draw[fill=blue!50,opacity=0.4] (B1) -- (A4) -- (A2) -- (B2) -- cycle;
\draw[->, dashed, thick,black] (3,0,0)--(3.5,0,0) node[right]{$\lambda_2$};
\draw[->, dashed, thick,black] (0,0,3)--(0,0,4) node[left]{$\lambda_1$};
\draw[->, dashed, thick,black] (0,3,0)--(0,3.5,0) node[above]{$\lambda_3$};
\draw[thick,black] (A4)--(B1)--(A7)--(A3)--(A2)--(A4);
\end{tikzpicture}}
\caption{The service rate regions of the systems in Example~\ref{ex:different_examples}.}
\label{fig:example1}
\end{figure}
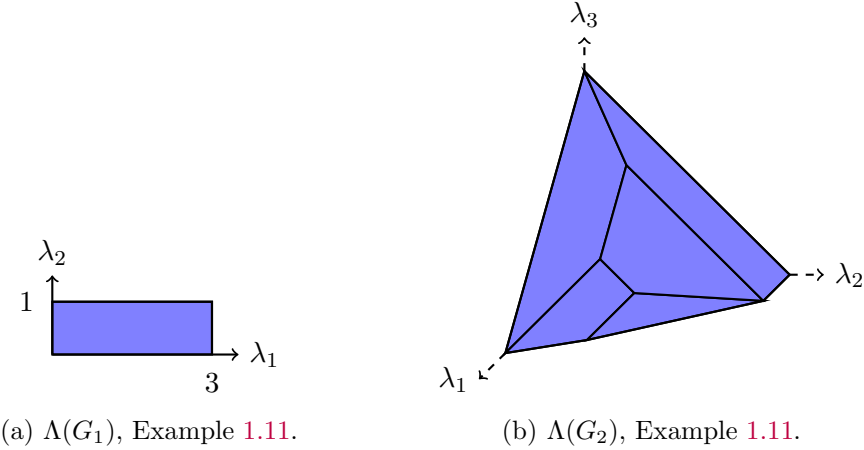

This paper mainly describes 
the geometric properties of the service rate region $\Lambda(\mR,\mu)$. Most of our results hold 
for an arbitrary $G$-system $\mR$, although our main focus is on $\Lambda(G)$.

%%%%%%%%%%%%%%%%%%%
%%%%%%%%%%%%%%%%%%%
%%%%%%%%%%%%%%%%%%%

\section{The Service Rate Region and the Allocation Polytopes}
\label{sec:2}

This section describes the polytope structure of the service rate region associated with an arbitrary $G$-system $\mR$. We also illustrate how the geometric structure has implications for the allocation of 
users in the corresponding system. 
For these purposes, viewing the service rate region as the image of a higher dimensional polytope under a linear map is often convenient, which we call the \textit{allocation polytope}.

\begin{definition} \label{def:all}
Let $\mR=(\mR_1,\ldots,\mR_k)$ be
a $G$-system, $m_i=|\mR_i|$ for $i\in \{1,\ldots,k\}$, and 
$m(\mR)=m_1 + \ldots +m_k$.
The \textbf{allocation polytope} of $(\mR,\mu)$ is the set of 
$(\lambda_{i,R} \mid i \in \{1, \ldots,k\}, \, R \in \mR_i)$
that satisfy the inequalities~\eqref{c2} and~\eqref{c3}.
We denote the allocation polytope by $$\mA(\mR,\mu) \subseteq \R^{m(\mR)}.$$
We also let $\mA(G)=\mA(\mR^\all(G)) = \mA(\mR^{\min}(G))$,
where the latter identity can be shown similarly to  Proposition~\ref{lostesso}.
\end{definition}

We now show that the allocation polytope is indeed a polytope, and state its connection with the service rate region. We will use the following fact, which easily follows 
from the definition of a convex hull combined with
Theorem~\ref{thm:fundamental_poly} and the observations right after it.

\begin{lemma}\label{lem:image_poly}
Let $\mP \subseteq \R^m$ be a polytope and let $f:\R^m \to \R^k$ be a linear map. Then $f(\mP)$ is a polytope and
$\vert(f(\mP)) \subseteq f(\vert(\mP))$.
\end{lemma}

The connection between the allocation polytope and the service rate region is described by the next result, which also summarizes some properties of these two regions. In particular, we are interested in maps of the form given in Theorem \ref{prop:def_f}.

\begin{theorem}\label{prop:def_f}
Let $\mR$ be a $G$-system and let $m(\mR)$ be as in Definition~\ref{def:all}. The following hold.
\begin{enumerate}
\item We have $f(\mA(\mR,\mu))=\Lambda(\mR,\mu)$,
where $f: \R^{m(\mR)} \to \R^k$
is the linear map defined by 
\begin{align}
    f: 
    \lambda=\left(\lambda_{i,R} \mid i \in \{1, \ldots,k\}, \, R \in \mR_i \right) &\mapsto \left(\sum_{R \in \mR_1}\lambda_{1,R},\ldots,\sum_{R \in \mR_k}\lambda_{k,R}\right). \nonumber
\end{align}
\item $\mA(\mR,\mu)$ and $\Lambda(\mR,\mu)$ are down-monotone polytopes. 
\end{enumerate}
\end{theorem}

\begin{proof}
The fact that $\Lambda(\mR,\mu)$ is the image of $\mA(\mR,\mu)$ under $f$ easily follows from Definitions~\ref{def:system} and~\ref{def:all}. 
We now establish the second part of the statement.
The set $\mA(\mR,\mu)$ is a polyhedron by definition. 
Its boundedness can be shown as follows.
Summing all the inequalities in~\eqref{c2} we get 
$\sum_{i=1}^k\sum_{R \in \mR_i}\lambda_{i,R}\leq n\mu$.
Using 
$\lambda_{i,R} \ge 0$ for all pairs~$(i,R)$,
we get that any $\lambda \in \mA(\mR,\mu)$ satisfies
\begin{equation*}
    \lVert\lambda\rVert_2\leq \sum_{i=1}^k\sum_{R \in \mR_i}|\lambda_{i,R}|=\sum_{i=1}^k\sum_{R \in \mR_i}\lambda_{i,R}\leq n\mu, 
\end{equation*}
where $\lVert\lambda\rVert_2$ is the 2-norm of $\lambda$. The fact that $\Lambda(\mR,\mu)$ is bounded follows from Lemma~\ref{lem:image_poly} and the 
boundedness of $\mA(\mR,\mu)$.
Finally, it is not hard to directly check that the polytopes $\mA(\mR,\mu)$ and $\Lambda(\mR,\mu)$ are down-monotone.
\end{proof}

We now turn to the natural question of describing the vertices and the points of the service rate region
with rational entries. We will use the connection between the service rate region and the allocation polytope to answer these questions.

We note that a linear map $f:\R^m \to \R^k$ is called \textbf{rational} if its matrix concerning the canonical basis has rational entries. The next result follows from Proposition~\ref{prop:inv} and Corollary~\ref{cor:rational_vert}.

\begin{lemma} \label{lem:rat_vert}
Let $\mP \subseteq \R^m$ be a rational polytope and $f:\R^m \to \R^k$ be a rational linear map. Then $f(\mP)$ is a polytope whose vertices have rational entries. 
\end{lemma}

The following result will be crucial to qualitatively describe the connection between the allocation polytope and the service rate region. 

\begin{lemma}\label{lem:rational_image}
$\mP \subseteq \R^m$ be a rational polytope and let $f:\R^m \to \R^k$ be a rational linear map. We have $f(\mP \cap \Q^m) = f(\mP) \cap \Q^k$.
\end{lemma}
\begin{proof}
The inclusion $\subseteq$ is immediate. To prove the other inclusion, we let $y \in f(\mP) \cap \Q^k$.
Write $f=(f_1,\ldots,f_k)$, where $f_i: \R^m \to \R$.
We need to show that there exists $x \in \mP \cap \Q^m$ with $f_i(x)=y_i$ for all $i \in \{1,\ldots,k\}$. Since $\mP$ is rational, 
$\mP=\{x \in \R^m \mid Ax^\top \le b^\top\}$ for some $A \in \Q^{\ell \times m}$ and $b \in \Q^\ell$.
We append to $A$ and $b$ a total of $2k$ rows, of which $k$ are for the inequalities $\{f_i(x) \le y_i \mid i=1, \ldots, k\}$ and the other $k$ are for the inequalities $\{-f_i(x) \le -y_i  \mid i=1, \ldots, k\}$.
Let $A'$ and $b'$ denote the resulting matrix and vector, of size $(\ell+2k) \times m$ and length $\ell+2k$, respectively. 
Then, by definition, $\mP'=\{x \in \R^m \mid A'x^\top \le b'^\top\}$ is a polyhedron.
Note that we need to prove that 
$\mP' \cap \Q^m \neq \emptyset$.
We first observe that~$\mP'$ is bounded because $\mP' \subseteq \mP$ and $\mP$ is bounded. Moreover, $\mP'$ is rational because $A'$ and $b'$ have rational entries (here, we use the fact that $y$ has rational entries and $f$ is a rational map).
Moreover, $\mP'$ is nonempty because~$y \in f(\mP)$ by assumption, which implies the existence of $z \in \mP$ with $f(z)=y$, i.e., of $z \in \mP'$. Then $\mP'$ has at least one vertex~$x$ and that vertex must have rational entries by Corollary~\ref{cor:rational_vert}.
Thus $\mP' \cap \Q^m \neq \emptyset$, as desired.
\end{proof}

We are now ready to state the main result of this section (note that the properties of the statement only hold for $\mu=1$).

\begin{theorem}\label{thm:properties_SRR}
Let $\mR$ be a $G$-system.
The following hold.
\begin{enumerate}
    \item The vertices of $\Lambda(\mR)$ have rational entries.
    \item  \label{ppp2} $\Lambda(\mR)\cap \Q^k=f(\mA(\mR)\cap \Q^k)$, where $f$ is defined as in Theorem~\ref{prop:def_f}.
\end{enumerate}
\end{theorem}
\begin{proof}
The vertices of $\Lambda(\mR)$ have rational entries because of Theorem~\ref{prop:def_f} and Lemma~\ref{lem:rat_vert}. The second part of the statement follows by combining Theorem~\ref{prop:def_f} with Lemma~\ref{lem:rational_image}.
\end{proof}

Note that Property~\ref{ppp2} of
Theorem~\ref{thm:properties_SRR}
implies that every rational point of the service rate region has a feasible rational allocation; see Definition~\ref{def:system}. This fact, which does not appear to be obvious,
is important from the application point of view 
in the sense of Subsection~\ref{subs:normal}.

\section{The Integer Allocation Model}
\label{sec:3}
In this section, we consider practical scenarios, described in Section~\ref{subs:normal}, wherein each server has a specific bandwidth, and each served request occupies the entire bandwidth when served, i.e., the $\lambda_{i,R}$'s are constrained to be either 0 or 1. We define the service rate region for this model and show how it relates to the model we considered in the previous section.

Assume we use the system $s\in \Z_{\ge 1}$ times, each time with possibly different allocation. Let $\alpha_i(R)$ be the number of times that recovery set $R\in \mR_i$ is used to recover the $i$-th object within the $s$ uses of the system. Then the number of times the $i$-th object is recovered is equal to $\lambda_i=\sum_{R\in \mR_i}\alpha_i(R)$. 
This motivates the following definitions.

\begin{definition}\label{def:allocations}
Let $\mR$ be a $G$-system.
An~$\mR$-\textbf{allocation} 
is a~$k$-tuple of functions~$\alpha=(\alpha_1,\dots,\alpha_k)$, where~$\alpha_i:\mR_i \to \N$ for all~$i \in \{1,\dots,k\}$.
The \textbf{service rate} of~$\alpha$ is the vector~$\lambda(\alpha)=(\lambda_1,\dots,\lambda_k) \in \N^k$, where 
$$\lambda_i=\sum_{R \in \mR_i} \alpha_i(R) \qquad \mbox{ for all $i \in \{1,\dots,k\}$}.$$
For $\nu \in \{1,\dots,n\}$ we define
$$\delta_\nu(\mR,\alpha) = \sum_{i=1}^k \; \sum_{R \in \mR_i} \delta_\nu(R) \, \alpha_i(R)~~
\text{where} ~~
\delta_\nu(R)= 
\begin{cases}
1 & \mbox{if $\nu \in R$}, \\
0 & \mbox{otherwise}.
\end{cases}$$
\end{definition}

In Definition \ref{def:allocations}, the quantity~$\delta_\nu(\mR,\alpha)$ represents the number of times server~$\nu$ is contacted.

\begin{definition}\label{def:servicerateregion}
Let $\mR$ be a $G$-system. 
The \textbf{one-shot service rate region} of $\mR$ with \textbf{capacity} $s \in \Z_{\ge 1}$
 is
$\Lambda_1(\mR,s)=\{\lambda(\alpha)/s \mid \alpha \mbox{ an $\mR$-allocation}, \,  \delta_\nu(\mR, \alpha) \le s \mbox{ for $1 \le \nu \le n$}\}$. 
The \textbf{rational service rate region} of~$\mR$ is the set
$$\Lambda^\Q(\mR)=\bigcup_{s \in \Z_{\ge 1} } \Lambda_1(\mR,s).$$
\end{definition}

In this section, we will show 
the following ``topological'' connection between the rational service rate region and the service rate region as defined in Section~\ref{sec:1}.

\begin{theorem} \label{main:discr}
Let $\mR$ be a $G$-system. The following hold.
\begin{enumerate}
    \item $\Lambda^\Q(\mR) = \Lambda(\mR) \cap \Q^k$.
    \item $\smash{\Lambda(\mR) = \overline{\Lambda^\Q(\mR)}}$, where the latter is the closure of $\Lambda^\Q(\mR)$ with respect to the Euclidean topology in $\R^k$.
\end{enumerate}
\end{theorem}

\begin{remark}
Before proving Theorem~\ref{main:discr}, we stress
that it is particularly relevant
for the practical scenarios described in Section~\ref{subs:normal}, which may require that allocations be integer or rational. It shows that 1) the rational points in the service rate region can be achieved with rational allocations, and 2) all points can be achieved by averaging over multiple system uses.
\end{remark}

We will use the following result in the proof of Theorem~\ref{main:discr}.

\begin{lemma}\label{lem:closure}
Let $\mP \subseteq \R^m$ be a down-monotone polytope. Then
$\smash{\mP=\overline{\mP \cap \Q^m}}$,
where the latter is the closure of 
$\mP \cap \Q^m$
for the Euclidean topology. 
\end{lemma}

\begin{proof}
The inclusion $\overline{\mP \cap \Q^m} \subseteq \mP$ holds because $\mP$ is closed and $\overline{\Q^m}=\R^m$, which implies
$\overline{\mP \cap \Q^m} \subseteq \overline{\mP} \cap \overline{\Q^m} = \mP \cap \R^m = \mP$. For the other inclusion, we will prove that for all $x \in \mP$ and all $\varepsilon>0$
we have 
$B_\varepsilon(x) \cap \mP \cap \Q^m \neq \emptyset$, where
$B_\varepsilon(x)$ is the ball of radius $\varepsilon$ centered at $x$.
This implies $\mP \subseteq  \overline{\mP \cap \Q^m}$ using, for example, \cite[Theorems 17.5 and 20.3]{Munkres}.
Fix any $x$ and $\varepsilon$ as above. Write $x=(x_1, \ldots, x_m)\in \R^m$. Since $\Q$ is dense in $\R$, for every $i \in \{1, \ldots, m\}$ there exists 
$y_i \in \Q$ with $x_i -\varepsilon/m \le y_i \le x_i$.
Since $\mP$ is down-monotone, we have $y=(y_1, \ldots, y_m) \in \mP$. 
The fact that
$y \in B_\varepsilon(x) \cap \mP \cap \Q^m$
now follows from
\begin{equation*}
    \left\| x-y \right\|_2 \le \left\| x-y \right\|_1 = \sum_{i=1}^m (x_i-y_i) \le m \cdot \varepsilon/m = \varepsilon,
\end{equation*}
We used the standard notation for the $p$-norm in $\R^k$.
\end{proof}

\begin{proof}[Proof of Theorem~\ref{main:discr}]
The second part of the statement follows from the first part in combination with Lemma \ref{lem:closure}.
Therefore it suffices to establish the first part.

Let $\lambda\in\Lambda^\Q(\mR)$. There is an $s\in\mathbb{Z}_{\ge 1}$ and an $\mR$-allocation $\alpha$ such that $s\lambda = \lambda(\alpha) = (\lambda_1,\ldots,\lambda_k)\in s\Lambda_1(\mR,s)$ and $\delta_\nu(\mR,\alpha)\leq s$, for all $\nu\in\{1,\ldots,n\}$. Note that $\lambda\in\Q^k$. We will show that the set 
$$\left\{\frac{\alpha_i(R)}{s} \, \mid \,  i\in\{1,\ldots,n\},\  R\in\mR_i \right\}$$ satisfies
Properties~\eqref{c1}, \eqref{c2}, and \eqref{c3}.
By definition, for $1 \le i \le k$ we have
$\lambda_i/s = \sum_{R\in\mR_i}\alpha_i(R)/s$.
Moreover, the condition $\delta_\nu(\mR,\alpha)\leq s$ can be rewritten as
\begin{equation*}\label{eq:not-exceed-capacity}
    \sum_{i=1}^k\sum_{\substack{R\in\mR_i\\ \nu\in R}} \frac{\alpha_i(R)}{s}\leq 1 \quad \textnormal{ for 
    $1 \le \nu \le n$}.
\end{equation*}
Finally, 
$\alpha_i(R)/s\geq 0$ for all
$i\in\{1,\dots, k\}$ and all $R\in\mR_i$.
We therefore conclude that $\lambda\in\Lambda(\mR) \cap \Q^k$, hence $\Lambda^\Q(\mR) \subseteq \Lambda(\mR) \cap \Q^k$. 

To prove the other containment, let $\lambda\in\Lambda(\mR) \cap \Q^k$. By Theorem \ref{thm:properties_SRR}, we have$\Lambda(\mR)\cap \Q^k=f(\mA(\mR)\cap \Q^k)$, where $f$ is defined in Theorem~\ref{prop:def_f}. In particular, there exist rational numbers 
$\{\lambda_{i,R}\in\Q \mid i\in\{1,\ldots,k\}, \; R \in \mR_i\}$, that satisfy Properties \eqref{c1}--\eqref{c3}
of Definition~\ref{def:system}. By definition, $\lambda_{i,R}=u_{i,R}/v_{i,R}$,  $u_{i,R},v_{i,R}\in\N$, and $v_{i,R}>0$ for all $i\in\{1,\ldots,k\}$ and $R \in \mR_i$. Let $s:=\mathrm{lcm}(v_{i,R} \mid i\in\{1,\ldots,k\}, \; R \in \mR_i)$, so that
$s\lambda =(s\lambda_1,\ldots,s\lambda_k)\in\N^k$. We claim that $s\lambda\in s\Lambda_1(\mR,s)$. 
Now for $1\leq i\leq k$, define 
$\alpha_i(R)= s \lambda_{i,R}$; note that $\alpha_i$ always maps to $\N$ by the construction of the number $s$.
Then for  $1\leq i \leq k$ we have 
$s\lambda_i = \sum_{R \in \mR_i}\alpha_i(R)$,
from which we conclude that 
$\alpha=(\alpha_1,\ldots,\alpha_k)$ is an $\mR$-allocation. Furthermore, for $\nu\in\{1,\ldots,n\}$ we have
\begin{align*}
    \delta_\nu(\mR,\alpha) &= \sum_{i=1}^k \sum_{R \in \mR_i}\delta_\nu(R)\alpha_i(R) 
    =  \sum_{i=1}^k \sum_{\substack{R \in \mR_i \\ \nu \in R}}\alpha_i(R) =  \sum_{i=1}^k \sum_{\substack{R \in \mR_i \\ \nu \in R}} s \lambda_{i,R} \leq s,
\end{align*} 
where the second equality follows from the definition of $\delta_\nu(R)$
and 
the last inequality follows from~\eqref{c2}. This shows that $s\lambda\in s\Lambda_1(\mR,s)$, equivalently $\lambda\in \Lambda_1(\mR,s)$, hence $\lambda\in\Lambda^\Q(\mR)$ as desired.
\end{proof}

We conclude this section with an example illustrating Theorem~\ref{main:discr}.

\begin{example}
\label{ex:7}
Let
$$G := \begin{pmatrix}
1 & 0 & 1 & 1 \\
0 & 1 & 1 & 2 
\end{pmatrix}\in\F_3^{2\times 4}.$$
Consider the $G$-system
$\mR=\mR^{\min}(G)=(\mR_1,\mR_2)$, where
$\mR_1:=\{\{1\}, \{2,3\}, \{2,4\},\{3,4\}\}$ and
$\mR_2:=\{\{2\}, \{1,3\}, \{1,4\},\{3,4\}\}$.
The corresponding service rate region is depicted in  Figure~\ref{fig:MDS}, along with the point
$P=(4/3, 2/3)$. 

\begin{figure}[hbt]
\centering
\begin{tikzpicture}[ thick,scale=0.8]

\coordinate (A1) at (0,0);
\coordinate (A2) at (0,5/2);
\coordinate (A3) at (5/2,0);
\coordinate (A4) at (2,1);
\coordinate (A5) at (1,2);

\draw[fill=blue!50,opacity=0.4] (A1) -- (A2) -- (A5) -- (A4) -- (A3) -- cycle;
\draw[->, thick,black] (0,0)--(3,0) node[right]{$\lambda_1$};
\draw[->, thick,black] (0,0)--(0,3) node[above]{$\lambda_2$};
\draw[thick,black] (A2) -- (A5) -- (A4) -- (A3);

\node[circle,scale=.5, label=below:$2.5$] (x4) at (2.5,0) {};

\node[circle,scale=.5,label=left:$2.5$] (y4) at (0,2.5) {};

\node[fill,circle,inner sep=1.5pt,label=$P$] at
(1.3333,0.6666){};
\end{tikzpicture}
\caption{Service rate region for the $G$-system in Example~\ref{ex:7} and the point $P=(4/3, 2/3)$.}\label{fig:MDS}
\end{figure}
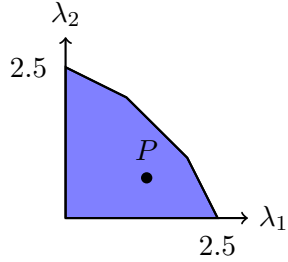

We have $P\in\Lambda(\mR)\cap\Q^2$
and $P\in \Lambda_1(\mR,3)$, i.e., $P$ can be achieved in three uses of the system. An example of an $\mR$-allocation, in the sense of Definition~\ref{def:allocations}, is given by
\begin{equation*}
\begin{aligned}
    \alpha_1: \mR_1 &\to \N \\
    \{1\} &\mapsto 2 \\
    \{2,3\} &\mapsto 1 \\
    \{2,4\}  &\mapsto 0 \\
    \{3,4\} &\mapsto 1 
\end{aligned}
    \qquad\qquad\qquad
    \begin{aligned}
    \alpha_2: \mR_2 &\to \N \\
    \{2\} &\mapsto 1 \\
    \{1,3\} &\mapsto 0 \\
    \{1,4\}  &\mapsto 1 \\
    \{3,4\} &\mapsto 0 
\end{aligned}
\end{equation*}

We have $\delta_1(\mR,\alpha) = \alpha_1(\{1\}) + \alpha_2(\{1,4\}) = 3$, $\delta_2(\mR,\alpha) = \alpha_1(\{2,3\}) + \alpha_2(\{2\}) = 2$, $\delta_3(\mR,\alpha) = \alpha_1(\{2,3\}) + \alpha_1(\{3,4\}) = 2$, $\delta_4(\mR,\alpha) = \alpha_1(\{3,4\}) + \alpha_2(\{1,4\}) = 2$. Moreover, $\sum_{R \in \mR_1} \alpha_1(R) =4$ and $\sum_{R \in \mR_2} \alpha_2(R) =2$, showing that $(4,2) \in \Lambda_1(\mR,3)$.
\end{example}

%%%%%%%%%%%%%%%%%%%
%%%%%%%%%%%%%%%%%%%
%%%%%%%%%%%%%%%%%%%

\section{Fundamental Parameters of the Service Rate Region}
\label{sec:4}

This section introduces some 
fundamental parameters of the service rate region that describe its ``shape''. We then 
study them by applying various techniques. The results are aimed at describing how the \textit{algebra} of the underlying matrix $G$
determines the \textit{geometry}
of the service region polytope 
$\Lambda(G,\mu)$.
To simplify the notation (and without loss of generality), we assume $\mu=1$. 
We start by recalling two types of elementary polytopes.

\begin{definition}
\label{def:hs}
Let~$h,\delta \in \R_{\ge 0}$.
The \textbf{$h$-hypercube} in~$\R^k$ is the convex hull of the set $\{x \in \R^k_{\ge 0} \mid x_i \in \{0,h\} \mbox{ for } 1 \le i \le k\}$. We say that~$h$ is the \textbf{size} of the hypercube. The \textbf{$\delta$-simplex} in~$\R^k$ is the convex hull of the set $\{\delta e_1, \ldots, \delta e_k\}$, where $e_i$ is the $i$-th standard basis vector of~$\F_q^k$.
Again, we say that $\delta$ is the \textbf{size} of the simplex. 
\end{definition}

We will introduce the first set of parameters for the service rate region. Other parameters will be introduced later.

\begin{definition}
\label{def:maxsum}
Let $\mR$ be a $G$-system.
We let:
\begin{alignat*}{4}
    \lambda^r(\mR) &= \textstyle \max\{\sum_{i=1}^k \lambda_i^r \mid \lambda \in \Lambda(\mR)\}, & & \qquad \textnormal{[\textbf{$r$-th max-sum capacity}]} \\
    \lambda(\mR) &= \lambda^1(\mR), & & \qquad \textnormal{[\textbf{max-sum capacity}]} \\
     \lambda_i^*(\mR) &= \max\left\{x \in \R \mid xe_i \in \Lambda(\mR) \right\}  \mbox{ for~$1 \le i \le k,$} & &
     \\ 
     \lambda^*(\mR) &= \max\left\{\lambda_i^*(\mR)  \mid 1 \le i \le k\right\}. & & 
\end{alignat*}
Furthermore, we denote as follows the largest size of a hypercube and a simplex contained in the service rate region:
\vspace{-0.3cm}
\begin{align*}
h(\mR) &= \max\{x \in \R \mid (x,\ldots,x) \in \Lambda(\mR) \} \\
\delta(\mR) &= \min\left\{\lambda_i^*(\mR) \mid 1 \le i \le k \right\}.
\end{align*}
When $\mR=\mR^\all(G)$
or $\mR= \mR^{\min}(G)$, we simply write
$\lambda^r(G)$, $\lambda(G)$,
$\lambda_i^*(G)$, $\lambda^*(G)$,
$h(G)$, and~$\delta(G)$.
\end{definition}

The next example shows that, in general, a point achieving the max-sum capacity will not achieve the$r$-th max-sum capacity for$r> 1$.

\begin{example}
Consider the service rate region of
Example~\ref{ex:7}, depicted in Figure~\ref{fig:MDS}. One can show that~$\lambda(G)$ is achieved by~$(1,2)$ and~$(2,1)$.
On the other hand,
$\lambda^2(G) = 6.25 > 5=1+4$ is achieved by~$(2.5,0)$ and~$(0,2.5)$. 
\end{example}

We start with a result showing how the parameters $h(\mR)$, $\lambda(\mR)$
and $\delta(\mR)$ relate to each other.

\begin{proposition}
\label{prop:hcube}
Let $\mR$ be a $G$-system. We have
$$ h(\mR) \le \min\left\{\frac{\lambda(\mR)}{k},{\delta(\mR)}\right\}.$$
\end{proposition}
\begin{proof}
We first show that~$h(\mR) \le \lambda(\mR)/k$. Suppose that~$h(\mR) > \lambda(\mR)/k$. By definition, we have $(h(\mR),\ldots,h(\mR)) \in \Lambda(\mR)$. Therefore $$\lambda(\mR) \ge h(\mR)k > \frac{\lambda(\mR)}{k}k = \lambda(\mR),$$ 
which is a contradiction. For the second part of the proof, assume that $h(\mR) > \delta(\mR).$ Then, by the definition of $\delta(\mR)$ there must exist at least one element of the set $\left\{h(\mR)e_i  \mid 1 \le i \le k \right\}\subseteq \F_q^k$ that 
does not belong to $\Lambda(\mR).$ This contradicts the definition of $h(\mR)$.
\end{proof}

\begin{remark}
\label{rem:notsharp}The bound of Proposition~\ref{prop:hcube} is met with equality for the service rate region depicted in Figure~\ref{subfig:1}.
However, the bound
is not sharp in general.
Consider for instance the service rate region~$\Lambda(G)$ for $$G = \begin{pmatrix}
1 & 1 & 0 & 1 \\
0 & 0 & 1 & 1 \\
0 & 0 & 0 & 1
\end{pmatrix}\in\F_2^{3\times 4}.$$
Note that
$(0,0,1),(0,1,0),(1,0,1),(2,1,0) \in \Lambda(G). $ 
It can be shown that $h(G) = 0.5$,
$\lambda(G)=3$, and $\delta(G)=1$. 
\end{remark}

In the next example, we show that the values~$\delta(\mR)$ and $\lambda(\mR)/k$ are not comparable in general, showing that taking the minimum in the bound of Proposition~\ref{prop:hcube} is indeed needed.

\begin{example}
For the service rate region of Example~\ref{ex:7}, we have~$\delta(G)=2.5 > 1.5 = \lambda(G)/2.$ However, for the service rate region of
Figure~\ref{subfig:2} we have $\delta(G)=1 < 2 = \lambda(G)/2.$ 
\end{example}

The quantity $\lambda^2(\mR)$ has a precise geometric significance; it gives the smallest sphere wedge that contains the service rate region.
To illustrate how $\lambda^2(\mR)$ relates to the other fundamental parameters, we will use an argument based on Bhatia-Davis inequality~\cite{bhatia2000better} from statistics.
Note that the following bound is sharp for the
$G$-system of Example~\ref{ex:simplex} below.

\begin{theorem}\label{thm:lambda2}
Let $\mR$ be a $G$-system. We have 
$$\lambda^2(\mR) \le \frac{k-1}{k}\lambda^*(\mR)\lambda(\mR) + \frac{(\lambda(\mR))^2}{k}.$$
\end{theorem}
\begin{proof}
Let~$\hat{\lambda} \in \Lambda(\mR)$ achieve~$\lambda^2(\mR)$. We apply the Bhatia-Davis inequality~\cite{bhatia2000better}
to the coordinates of~$\hat{\lambda}$, obtaining 
\begin{align*}
\frac{1}{k}\sum_{i=1}^k {\hat{\lambda}_i}^2 &\le \left(\max_{1 \le i \le k}\{\hat{\lambda}_i\}-  \frac{1}{k}\sum_{i=1}^k \hat{\lambda}_i\right)\left(\frac{1}{k}\sum_{i=1}^k \hat{\lambda}_i - \min_{1 \le i \le k}\{\hat{\lambda}_i\}\right) + \frac{1}{k^2}\left(\sum_{i=1}^k \hat{\lambda}_i\right)^2 \\
&\le  \left(\frac{k-1}{k} \max_{1 \le i \le k}\{\hat{\lambda}_i\}\right)\left(\frac{1}{k}\sum_{i=1}^k \hat{\lambda}_i \right)+ \frac{1}{k^2}\left(\sum_{i=1}^k \hat{\lambda}_i\right)^2 \\
&\le \left(\frac{k-1}{k}\lambda^*(\mR)\right)\left(\frac{1}{k}\lambda(\mR)\right) + \frac{1}{k^2}\lambda(\mR)^2.
\end{align*}
Since~$\hat{\lambda}$ achieves~$\lambda^2(\mR)$ by assumption, we can rewrite the inequality we just obtained as follows:
\[\frac{1}{k}\lambda^2(\mR) \le \frac{k-1}{k^2}\lambda^*(\mR)\lambda(\mR) + \frac{1}{k^2}(\lambda(\mR))^2.\]
Multiplying both sides by~$k$ gives the desired result.
\end{proof}

Another natural parameter of the service rate region is its volume. Recall that
the volume 
of a convex polytope $\mP$
is the Lebesgue measure~\cite{lang1986note} of its interior, which we denote by~$\vol(\mP)$.
Distributed service systems strive to support data download of simultaneous users whose numbers and interests vary over time. The larger the service rate region volume, the larger the number of different user-number configurations the system can serve.

Computing the volume of a polytope is a difficult task in general~\cite{furedi1986computing}. However, some cases are relevant for our purposes where simple observations give a closed formula for the volume of the service rate region.

\begin{proposition}
\label{prop: volser}
Suppose that 
$G$ is a replication matrix; see page~\pageref{pagerepl} for the definition.
We have
$$\vol(\Lambda(G)) =  \prod_{i=1}^k |\{1 \le \nu \le n \mid \mbox{the $\nu$-th column of $G$ is a nonzero multiple of $e_i$}\}|.$$

\end{proposition}
\begin{proof}
It is easy to see that the service rate region $\Lambda(G)$ 
is a hyperrectangle in~$\R^k$, where each edge has a length equal to the number of times the corresponding standard basis vector appears as a column in $G$. The $\Lambda(G)$ volume is then determined as the quantity in the statement.
\end{proof}

In Theorem~\ref{thm:volume3D} we will give a closed formula for the volume of $\Lambda(G)$, when $G$ generates a 3-dimensional MDS code of length at least 6. The result is embedded in  Section~\ref{sec:6}, devoted to the service rate region of systematic MDS codes.

We now compute the volume of the allocation polytope of a replication system, showing in particular that the volume of the allocation polytope does not determine the volume of the service rate region. Intuitively, this follows from the fact that the volume of the allocation polytope is multilinear in the coordinates corresponding to the same object, whereas the volume of the service rate polytope is linear in their sum.
We introduce a class of polytopes that will be used later in Section~\ref{sec:5}.

\begin{definition}
\label{def:knapsack}
A polytope of the form
$\mP=\{x \in [0,1]^m \mid y x^\top \le n\} \subseteq \R^m$,
where $n$ and~$m$ are positive integers and~$y \in \R_{\ge 0}^m$ is a vector,
is called a 
\textbf{relaxed knapsack polytope} in $\R^m$. 
\end{definition}

The volume of a relaxed knapsack polytope 
as in
Definition~\ref{def:knapsack}
is known to be
\begin{equation} \label{eq:volk}
    \vol(\mP) = \frac{1}{m!\prod_{i=1}^m y_i} \ \sum_{x \in \{0,1\}^m \cap \mP} (-1)^{\wt(x)}g(x)^m,
\end{equation}
where $\wt(x)$ is the number of nonzero entries of $x$ and $g(x) = n - \sum_{i=1}^m y_ix_i$ for all $x \in \R^m$; see e.g.~\cite{barrow1979spline}.
Using the above formula for the volume and some elementary generating functions theory,
we compute the volume of the allocation polytope of a replication matrix.

\begin{proposition}
\label{prop:volumeallo}
Suppose that $G$ is a replication matrix. We have $\vol(\mA(G))=1$.
\end{proposition}
\begin{proof}
It is not hard to see that $\mA(G)=\{x \in [0,1]^n \mid x_1 + \cdots +x_n \le n\}
$, which 
is a relaxed  knapsack polytope obtained for $m=n$ and $y = 1^n=(1, \ldots,1)$. Therefore, using~\eqref{eq:volk}
and denoting by $[x^n]S(x)$ the coefficient of $x^n$ in a power series~$S(x)$,
we compute
\allowdisplaybreaks
\begin{align}
\vol(\mA(G)) &=  
% \frac{1}{n!}\sum_{x \in \{0,1\}^n \cap P} (-1)^{\wt(x)}g(x)^n \label{eqn:a1} \\
% &= 
\frac{1}{n!}\sum_{x \in \{0,1\}^n} (-1)^{\wt(x)}(n-(x_1+\ldots+x_n))^n = \frac{1}{n!}\sum_{x \in \{0,1\}^n} (-1)^{\wt(x)}(n-\wt(x))^n  \nonumber\\
&= \frac{1}{n!}\sum_{i=0}^n (-1)^{i}(n-i)^n\binom{n}{i}= \sum_{i=0}^n (-1)^{i}\binom{n}{i}[x^n]e^{(n-i)x} = [x^n](e^x-1)^n  = 1,  \nonumber
\end{align}
where 
all passages easily follow from
Binomial Theorem and the
Taylor expansion of the exponential function. 
\end{proof}

Note that even though the volume of the allocation polytope is the same for every replication matrix, the volume of the service rate region is not a constant; see Proposition~\ref{prop: volser}.

Throughout this section,
we focus on the connection between 
the parameters of~$\Lambda(G)$ and those of the error-correcting code generated by $G$; see Appendix~\ref{sec:appendix2} for some coding theory background.

We start by recalling the following result from~\cite{service:aktas2021jkks}, whose statement
relies on interpreting the columns of $G$ as points of the finite projective space $\PG(k-1,q)$; see~\cite{MR1186841} for a general reference. This can be done because, as stated in Notation~\ref{not:whatG}, none of the columns of $G$ is the zero vector.

\begin{proposition}
\label{prop:emina}
Let~$\lambda \in \Lambda(G)$ and let $I \subseteq \{1,\ldots,k\}$ be an index set. Let $H$ be a hyperplane of~$\PG(k-1,q)$ not containing any of the standard basis vectors $e_i$, for~$i \in I$,
 and let $S$ denote the multiset of columns of~$G$ in~$\PG(k-1,q)$. We have 
$$\sum_{i \in I}  \lambda_i \le |S \setminus H|,$$
where $S\setminus H$ is the multiset of points obtained from $S$ after removing all the points contained in~$H$, counted with their multiplicity.
\end{proposition}

The following lemma is well-known and can be shown by considering the columns of~$G$ as a multiset of points in $\PG(k-1,q)$, which are not all contained in a hyperplane since $G$ has rank~$k$ by assumption.

\begin{lemma}
\label{lem:PG}
Let $S$ be the multiset of the
columns of~$G$, viewed as projective points in $\PG(k-1,q)$. Let $d$ be the minimum distance of the code generated by $G$.
Then every hyperplane of $\PG(k-1,q)$ contains at most $n-d$ points of $S$, and there exists a hyperplane of $\PG(k-1,q)$ which contains exactly $n-d$ points of $S$.
\end{lemma}

We can also apply Proposition~\ref{prop:emina} to show a connection between the minimum distance of the code generated by~$G$ and the largest simplex contained in the service rate region.

\begin{corollary}\label{cor:simplex}
Let $d$ denote the minimum distance of the code generated by $G$. We have
$$\lceil \delta(G)\rceil \le d.$$
\end{corollary}

\begin{proof}
Let $i \in \{1, \ldots, k\}$
be fixed and let $H\subseteq \PG(k-1,q)$ be a hyperplane that does not contain~$e_i$.  By applying Proposition \ref{prop:emina} with~$I =\{i\}$ and~$\lambda = \delta(G)e_i \in \Lambda(G)$, we have
$$ \delta(G)\leq | S\setminus H| = n- |S\cap H|,$$
where~$S$ is the multiset of columns of~$G$.
We can follow the same reasoning for every $i$. Since every hyperplane $H$ does not contain some $e_i$, 
$$d = n-\max\{| S\cap H| \, : \,  H\subseteq \PG(k-1,q), \, H\mbox{ hyperplane}\}$$ and~$\delta(G) \leq | S\setminus H|$, we conclude that $\lceil \delta(G) \rceil \leq d$.
\end{proof}

The bound of Corollary~\ref{cor:simplex} is met with equality
by some matrices $G$, as the following example illustrates.

\begin{example}
\label{ex:otherMDS}
Let  $$G = \begin{pmatrix} 1 & 0 & 0 & 1 \\ 0 & 1 & 0 & 1 \\ 0 & 0 & 1 & 1 \end{pmatrix} \in \F_2^{3 \times 4}$$ It can be easily seen that~$2e_i \in \Lambda(G)$ for $i \in \{1,2,3\}$ and therefore $\delta(G) = 2$.
\end{example}

In the last part of this section, inspired by the coding theory literature, we introduce the notion of \textit{availability} for the matrix $G$.
We then describe the role this notion plays in shaping the geometry of the service rate region.

\begin{definition}
Suppose that $G$ is systematic. We say that $G$ has \textbf{availability}
$t\in \Z_{\ge 0}$ if $\mR_i^\all(G)$ contains $t+1$ pairwise disjoint sets for all $i \in \{1, \ldots, k\}$.
\end{definition}

The following result easily follows from the definitions.

\begin{proposition} \label{prop:avail2}
Suppose that $G$ is systematic and has availability $t$. Then $(t+1)e_i \in \Lambda(G)$ for all $i \in \{1, \ldots, k\}$.
In particular, $\lfloor\delta(G)\rfloor \ge t+1$.
\end{proposition}

By combining
Corollary~\ref{cor:simplex} with
Proposition~\ref{prop:avail2} we obtain the following result.

\begin{corollary} \label{cor:boundd}
Suppose that $G$ is systematic and has availability $t$. Let $d$ denote the minimum distance of the code generated by $G$. We have $d \ge t+1$.
\end{corollary}

We conclude this section 
with an example where
Proposition~\ref{prop:avail2}
and Corollary~\ref{cor:boundd} are sharp.

\begin{example}[The simplex code]
\label{ex:simplex}  
Let
$$G = \begin{pmatrix}
1 & 0 & 0 & 1 & 1 & 0 & 1 \\
0 & 1 & 0 & 1 & 0 & 1 & 1 \\
0 & 0 & 1 & 0 & 1 & 1 & 1 \\
\end{pmatrix} \in \F_2^{3 \times 7}.$$
Then $G$ has availability 3
and Proposition~\ref{prop:avail2} is sharp in this case.
Note that $G$ is the generator matrix of one of the best-known error-correcting codes, namely the simplex code; see e.g.~\cite{macwilliams1977theory}.
\end{example}

\section{Outer Bounds}
\label{sec:5}

In this section, we derive outer bounds for the service rate region $\Lambda(G)$ as bounding polytopes $\mP \supseteq \Lambda(G)$. We apply methods
from coding theory and optimization, dedicating a subsection to each of the two approaches. We illustrate how to apply the bounds with examples and comment on their sharpness.

\subsection{Coding Theory Approach}
We start with a simple result 
that can be easily obtained by summing the inequalities that define the allocation polytope, namely the constraints in \eqref{c2} for $1 \le \nu \le n$.

\begin{lemma}[Total Capacity Bound] \label{tcb}
Let $\mR$ be a $G$-system and let $\{\lambda_{i,R}\}$ be a feasible allocation for $(\mR,\mu)$; see Definition~\ref{def:system}.
We have
\begin{align} \label{new1}
\sum_{i=1}^k \sum_{R\in\mR_i}  |R|\lambda_{i,R} \le \mu n.
\end{align}
\end{lemma}

The following result links the size of the recovery sets of a $G$-system to the parameters of the (dual of the) error-correcting code
generated by $G$.

\begin{proposition}
\label{cor:size_recovery}
Suppose that~$G$ is systematic. Let
$d^\perp$ denote the minimum distance of the dual 
of the code generated by $G$.
For all $i\in\{1,\ldots, k\}$ and
$R\in\mR^\all_i(G)$ we have
$R=\{i\}$ or~$|R|\geq d^\perp-1$.
\end{proposition}

We now establish the first outer bound of this section.

\begin{theorem}[Dual Distance Bound]\label{thm:bounddualdis}
Suppose that~$G$ is systematic.
Let
$d^\perp$ denote the minimum distance of the dual 
of the code generated by $G$. If~$(\lambda_1,\dots,\lambda_k) \in \Lambda(G)$, then
$$\sum_{i=1}^k \Bigl( \min\{\lambda_i,1\} + (d^\perp-1) \max\{0,\lambda_i-1\}  \Bigr) \le n.$$
\end{theorem}
\begin{proof}
Let $(\lambda_1,\dots,\lambda_k) \in \Lambda(G)$ and let
$\{\lambda_{i,R}\}$ be a feasible allocation for $(\mR,1)$.
By Proposition~\ref{cor:size_recovery} we have~$|R|\geq d^\perp-1$ for every $i\in\{1,\ldots, k\}$ and~$R\in\mR_i^\all(G)$ with $R\ne\{i\}$. We can therefore rewrite the LHS of \eqref{new1} as follows:
\allowdisplaybreaks
\begin{align} \allowdisplaybreaks \label{ee2}
\sum_{i=1}^k \lambda_{i,\{i\}} + \sum_{i=1}^k \sum_{\substack{R\in\mR_i^\all(G)\\ R \ne \{i\}}} |R| \lambda_{i,R}   \ &\geq \ 
\sum_{i=1}^k \lambda_{i,\{i\}} + ( d^\perp-1) \sum_{i=1}^k \sum_{\substack{R\in\mR_i^\all(G)\\ R \ne \{i\}}}\lambda_{i,R}\nonumber\\ &= \ 
 \sum_{i=1}^k \lambda_{i,\{i\}} + ( d^\perp-1) \sum_{i=1}^k \left(\lambda_i-\lambda_{i,\{i\}}\right)
\nonumber\\ &= \ 
(d^\perp-1) \sum_{i=1}^k\lambda_i-(
d^\perp-2) \sum_{i=1}^k \lambda_{i,\{i\}}.
\end{align}
Since $G$ has no all-zero column, we have $d^\perp \ge 2$. Therefore, using the fact that 
$\lambda_{i,\{i\}} \le \min\{\lambda_i,1\}$ for all~$i$, we can further say that the right-hand side of~\eqref{ee2} is at least
\begin{align*}
&(d^\perp-1)\sum_{i=1}^k\lambda_i-(d^\perp-2) \sum_{i=1}^k\min\{\lambda_i,1\} =\sum_{i=1}^k \Bigl( \min\{\lambda_i,1\} + (d^\perp-1)  \max\{0,\lambda_i-1\}  \Bigr),
\end{align*}
which, combined with \eqref{new1},
gives the statement.
\end{proof}

\begin{remark}
\label{rem:sharp}
It follows from~\cite{service:AndersonJJ18} that 
the Dual Distance Bound of Theorem~\ref{thm:bounddualdis} is sharp if $G$ is a systematic MDS matrix and $n \ge 2k$; see Appendix~\ref{sec:appendix2} for the definition of an MDS matrix.
The bound can be sharp also for systematic matrices $\smash{G \in \F_q^{k \times n}}$ that generate an MDS code and have
$n < 2k$. This is the case of the matrix $G$ of
Example~\ref{ex:otherMDS}. 
\end{remark}

It turns out that Theorem~\ref{thm:bounddualdis} is not particularly effective
for systems that mainly implement replication, i.e., for matrices $G$ that are very similar to a replication matrix. We obtain the following result by considering the number of systematic nodes for each object. 
Since the proof is similar to the one of Theorem~\ref{thm:bounddualdis}, we omit it here.

\begin{theorem}
\label{thm:2distance}
Suppose that $G$ is systematic and let $s_i$ denote the number of systematic nodes for the~$i$-th object, for
$i \in \{1, \ldots, k\}$. If $(\lambda_1,\dots,\lambda_k) \in \Lambda(G)$, then
$$\sum_{i=1}^k \Bigl( \min\{\lambda_i,s_i\} + 2\max\{0,\lambda_i-s_i\}  \Bigr) \le n.$$
\end{theorem}

The outer bounds 
Theorems~\ref{thm:bounddualdis} and~\ref{thm:2distance} are not generally comparable, as the following example illustrates.

\begin{example}
\label{ex:referee}
An example where Theorem~\ref{thm:bounddualdis} outperforms 
Theorem~\ref{thm:2distance} is given by the region in Example \ref{ex:otherMDS}. By Remark \ref{rem:sharp}, the bound of Theorem \ref{thm:bounddualdis} gives the exact service rate region. This automatically outperforms the bound of Theorem~\ref{thm:2distance} as $(d^\perp-1)=3 > 2$. 
Now consider the service rate region of Example~\ref{ex:different_examples} depicted in Figure~\ref{subfig:2}. 
The bounding polytopes given by Theorems~\ref{thm:bounddualdis} and~\ref{thm:2distance} are depicted in Figure~\ref{fig:compare}, showing that 
Theorem~\ref{thm:2distance} outperforms 
Theorem~\ref{thm:bounddualdis} in that case.
\end{example}

\begin{figure}[hbt]
\centering

\begin{tikzpicture}[thick,scale=1.0]
\coordinate (A1) at (0,0);
\coordinate (A2) at (0,1);
\coordinate (A3) at (3,0);
\coordinate (A4) at (3,1);
\coordinate (B1) at (0,4);
\coordinate (B2) at (0,2.5);
\coordinate (B3) at (4,0);
\coordinate (B4) at (3.5,0);

\node at (B1) [left = 1mm of B1] {4};
\node at (B2) [left = 1mm of B2] {$2.5$};
\node at (B4) [below = 1mm of B4] {$3.5$};
\node at (B3) [below = 1mm of B3] {4};

\node at (A2) [left = 1mm of A2] {1};
\node at (A3) [below = 1mm of A3] {3};
\node at (A4) [right = 1mm of A4] {$(3,1)$};

\draw[->, thick,black] (0,0)--(4.5,0) node[right]{$\lambda_1$};
\draw[->, thick,black] (0,0)--(0,4.5) node[above]{$\lambda_2$};
\draw[thick,blue] (0,1)--(3,1)--(3,0);
\draw[thick,cyan] (0,4)--(4,0);
\draw[thick,gray] (0,2.5)--(3,1)--(3.5,0);

\draw[fill=gray!50] (0,1)--(0,2.5)--(3,1)-- cycle;

\draw[fill=gray!50] (3,1)--(3,0)-- (3.5,0) -- cycle;

\draw[fill=blue!50,opacity=0.4] (0,1)--(3,1)--(3,0) -- (0,0) -- cycle;

\draw[fill=cyan!50] (0,2.5)--(0,4) -- (3,1) -- cycle;

\draw[fill=cyan!50] (3,1)--(3.5,0) -- (4,0) -- cycle;

 \path (2, 4) node[draw, scale = 0.9, anchor=north west, text width=9em]
  {%
  
 \tikz{\draw[cyan, fill=cyan!50] (0, 0) rectangle (1em, 1.5ex);} \, Theorem~\ref{thm:bounddualdis} 
  
   \tikz{\draw[gray, fill=gray!50] (0, 0) rectangle (1em, 1.5ex);} 
 \, Theorem~\ref{thm:2distance} 
  
  \tikz{\draw[blue, fill=blue!50, opacity=0.4] (0, 0) rectangle (1em, 1.5ex);}  \, $\Lambda(G)$  
  };

\end{tikzpicture}
    \caption{The service rate region of Example~\ref{ex:different_examples}, Figure~\ref{subfig:2}, is an example where the bound of Theorem~\ref{thm:2distance} gives a better approximation than the bound of Theorem~\ref{thm:bounddualdis}.} 
    \label{fig:compare}
\end{figure}
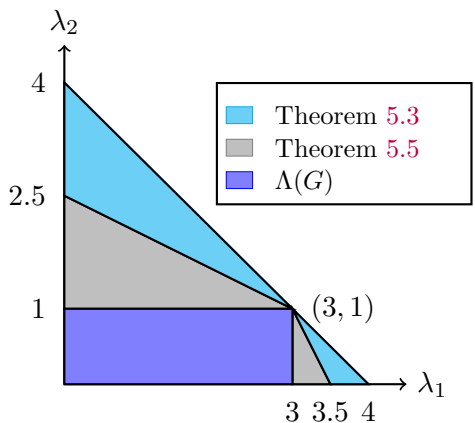

The next result is a hybrid between Theorems~\ref{thm:bounddualdis} and~\ref{thm:2distance}, in the sense that it takes into account \textit{both} the 
minimum distance of the dual of the code
generated by $G$
and the number of systematic nodes.

\begin{theorem}
\label{thm:sysDDB}
Let $d^\perp$ be the minimum distance of the dual of the code generated by $G$ and let~$s_i$ denote the number of systematic nodes for the $i$-th object, for $i \in \{1, \ldots, k\}$. For all~$(\lambda_1,\dots,\lambda_k) \in \Lambda(G)$ we have
$$\sum_{\substack{i \in \{1,\ldots,k\}\\ s_i \neq 0}} \left( \min\{s_i,\lambda_i\} +\max\{2,d^\perp -1\}\max\{0,\lambda_i - s_i\}\right) + \sum_{\substack{i \in \{1,\ldots,k\} \\ s_i = 0}} 2\lambda_{i}\le n.$$
\end{theorem}
\begin{proof}
Let $(\lambda_1,\dots,\lambda_k) \in \Lambda(G)$ and let
$\{\lambda_{i,R}\}$ be a corresponding feasible allocation. 
Write $\mR$ for $\mR^{\min}(G)$.
By Proposition~\ref{cor:size_recovery} and the fact that~$d^\perp = 2$ if there exist two columns of~$G$ that are linearly dependent, we obtain 
\allowdisplaybreaks
\begin{multline*}
\sum_{i \in \{1,\ldots,k\}} \sum_{R \in \mR} |R|\lambda_{i,R} \ge  \sum_{\substack{i \in \{1,\ldots,k\}  \\ s_i\neq 0}} \left( \sum_{\substack{R\in\mR \\ |R| = 1}} \lambda_{i,R}+ \max\{2,d^\perp -1\} \sum_{\substack{R\in\mR \\ |R| \ne 1}} \lambda_{i,R}\right) + \sum_{\substack{i \in \{1,\ldots,k\} \\ s_i = 0}} 
\sum_{R\in\mR} 
2\lambda_{i,R}.    
\end{multline*}
Therefore, 
\begin{align*}
\sum_{i \in \{1,\ldots,k\}} & \sum_{R \in \mR} |R|\lambda_{i,R} \\
&\ge \sum_{\substack{i \in \{1,\ldots,k\} \\ s_i \neq 0}} \sum_{\substack{R\in\mR \\ |R| = 1}} \lambda_{i,R} + \sum_{\substack{i \in \{1,\ldots,k\} \\ s_i \neq 0}} \left(\max\{2,d^\perp -1\}\left( \lambda_i - \sum_{\substack{R\in\mR \\ |R| = 1}} \lambda_{i,R}\right)\right) + \sum_{\substack{i \in \{1,\ldots,k\} \\ s_i = 0}} 2\lambda_{i} 
\\ 
&\ge \sum_{\substack{i \in \{1,\ldots,k\}  \\ s_i \neq 0}} \max\{2,d^\perp -1\}\lambda_i - \left( \max\{2,d^\perp -1\} -1\right) \sum_{\substack{i \in \{1,\ldots,k\}  \\ s_i \neq 0}} \min\{s_i,\lambda_i\} + \sum_{\substack{i \in \{1,\ldots,k\} \\ s_i = 0}} 2\lambda_{i} 
\\
&= \sum_{\substack{i \in \{1,\ldots,k\} \\ s_i \neq 0}} \left( \min\{s_i,\lambda_i\} +\max\{2,d^\perp -1\}\left(\lambda_i - \min\{s_i,\lambda_i\} \right)\right) + \sum_{\substack{i \in \{1,\ldots,k\}  \\ s_i = 0}} 2\lambda_{i}\\
&= \sum_{\substack{i \in \{1,\ldots,k\} \\ s_i \neq 0}} \left( \min\{s_i,\lambda_i\} +\max\{2,d^\perp -1\}\max\{0,\lambda_i - s_i\}\right) + \sum_{\substack{i \in \{1,\ldots,k\} \\ s_i = 0}} 2\lambda_{i},
\end{align*}
where the first inequality follows from Definition~\ref{def:system}, and the second follows from the inequality
\[
\sum_{\substack{R\in\mR \\ |R| = 1}} \lambda_{i,R} \le \min\{s_i,\lambda_i\}.
\]
\end{proof}

It is interesting to note that if
$d^\perp -1 \ge 2$ (and hence $s_i=1$ for all $i\in\{1,\ldots,k\}$),
then Theorem~\ref{thm:sysDDB} gives Theorem~\ref{thm:bounddualdis}. Similarly, if~$s_i \neq 0$ for all~$i \in \{1,\ldots,k\}$ and~$s_i \ge 2$ for at least one~$i$, then Theorem~\ref{thm:sysDDB} becomes Theorem~\ref{thm:2distance}.

Lemma~\ref{tcb} suggests that the variety of sizes of the recovery sets plays an important role in shaping the service rate region.
By taking into account the 
indices $i$ for which the recovery sets
All have the same size, so we obtained the following result. Note that this exactly measures the contribution of the indices for which the recovery sets have the same size and thus improves upon Theorem \ref{thm:sysDDB}.
The proof
is a simple extension of Theorem~\ref{thm:sysDDB}, and we omit it here.

\begin{theorem}
\label{var=0}
Let $d^\perp$ be the minimum distance of the dual of the code generated by $G$ and let~$s_i$ denote the number of systematic nodes for the $i$-th object, for $1 \in \{1, \ldots, k\}$.
Let 
\begin{align*}
\mu_i &= \frac{1}{|\mR^{\min}_i(G)|-s_i} \;  \sum_{\substack{R\in\mR^{\min}_i(G) \\ |R| \neq 1}} |R| \quad \mbox{for $i \in \{1, \ldots, k\}$,} \\
J &= \{i \in \{1, \ldots, k\} \mid 
\mbox{all $R \in \mR^{\min}_i(G)$ with $|R| \neq 1$ have the same cardinality}\}.
\end{align*}
Then for all
$\lambda \in \Lambda(G)$ we have
\begin{align*}
n &\ge \sum_{\substack{i \in \{1,\ldots,k\}\\ s_i \neq 0,\; i \notin J}} \left( \min\{s_i,\lambda_i\} +\max\{2,d^\perp -1\}\max\{0,\lambda_i - s_i\}\right) + \sum_{\substack{i \in \{1,\ldots,k\} \\ s_i = 0,\; i \notin J}} 2\lambda_{i} \\
& \qquad \quad+ \sum_{\substack{i \in \{1,\ldots,k\} \\ s_i = 0,\; i \in J}} \mu_i\lambda_{i} + \sum_{\substack{i \in \{1,\ldots,k\} \\ s_i \neq 0,\ |\mR_i^{\min}(G)| =s_i }} s_{i} +  \sum_{\substack{i \in \{1,\ldots,k\} \\ s_i \neq 0,\; i \in J}} \left(\mu_i\lambda_i - (1-\mu_i)\min\{s_i,\lambda_i\}\right).
\end{align*}
\end{theorem}

In the next example, we show that Theorem~\ref{var=0} can be sharper than Theorem~\ref{thm:sysDDB} for some service rate regions.

\begin{example}
\label{ex:comparison}
Let $k=3$, $n=6$, $q=3$, and 
$$G := \begin{pmatrix}
0 & 1 & 1 & 2 & 1 & 2  \\
1 & 2 & 2 & 2 & 1 & 1 \\
0 & 0 & 0 & 1 & 2 & 2
\end{pmatrix}\in\F_3^{3\times 6}.$$ 
Following the notation of
Theorem~\ref{var=0}
we have
$d^\perp=2$, $(\mu_1,\mu_2,\mu_3)=(2,3,11/4)$,
$(s_1,s_2,s_3)=(0,1,0)$, and
$J=\{1,2\}$. 
Figure~\ref{fig:comparison} depicts the service rate region $\Lambda(G)$
and the outer bounds given by Theorems~\ref{thm:sysDDB} and~\ref{var=0}.
\end{example}

\begin{figure}[hbt!]
    \centering
\begin{tikzpicture}[thick,scale=0.9]
\coordinate (D1) at (0,0,0);
\coordinate (D2) at (0,2,0);
\coordinate (D3) at (0,0,2);
\coordinate (D4) at (2,0,0);

\coordinate (A1) at (0,3,0);
\coordinate (A2) at (0,0,3);
\coordinate (A3) at (8/3,0,0);
\coordinate (A4) at (1,0,5/2);
\coordinate (A5) at (1,5/2,0);

\coordinate (A6) at (7/2,0,0);

\begin{scope}[thick,dashed,,opacity=0.6]
\draw (D1) -- (D2);
\draw (D1) -- (D3);
\draw (D1) -- (D4);

\end{scope}

\draw[fill=red,opacity=0.4] (A1) -- (A2) -- (A4) -- (A5) -- cycle;
\draw[fill=red,opacity=0.4] (A4) -- (A3) -- (A5)  -- cycle;

\draw[fill=cyan,opacity=0.4] (A1) -- (A2) -- (A4) -- (A5) -- cycle;
\draw[fill=cyan,opacity=0.4] (A4) -- (A6) -- (A5) -- cycle;

\draw[fill=blue!50,opacity=0.4] (D2) -- (D3) -- (D4) -- cycle;

\draw[->, dashed, thick,black] (0,0,0)--(4,0,0) node[right]{$\lambda_2$};
\draw[->, dashed, thick,black] (0,0,0)--(0,0,4) node[left]{$\lambda_1$};
\draw[->, dashed, thick,black] (0,0,0)--(0,4,0) node[above]{$\lambda_3$};

\path (2, 4) node[draw, scale = 0.9, anchor=north west, text width=9em]
  {%
     \tikz{\draw[cyan, fill=cyan!50] (0, 0) rectangle (1em, 1.5ex);} \, Theorem~\ref{thm:sysDDB}
    
     \tikz{\draw[gray, fill=gray!50] (0, 0) rectangle (1em, 1.5ex);} \, Theorem~\ref{var=0}
    
     \tikz{\draw[blue, fill=blue!50, opacity=0.4] (0, 0) rectangle (1em, 1.5ex);} \, $\Lambda(G)$ 
  };

\end{tikzpicture}
    \caption{Service rate region and outer bounds for  Example~\ref{ex:comparison}.}
    \label{fig:comparison}
\end{figure}
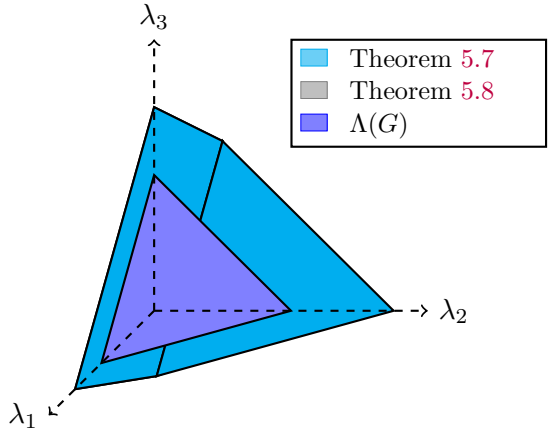

\subsection{Optimization Approach}

In this subsection, we use the theory of knapsack polytopes (recall Definition~\ref{def:knapsack}) to derive an outer bound for the allocation and service rate region polytopes and the bound corollaries. Most notably, we obtain upper bounds for the quantities 
$\sum_{i \in I} \lambda_i$, for any given index set $I \subseteq \{1, \ldots, k\}$. These quantities are of interest in practice because they correspond to the cumulative numbers of users interested in some (sub)sets of stored objects. Observe that when $I = \{1, \ldots, k\}$, then $\sum_{i \in I} \lambda_i$ represents the total number of users in the system.
In this subsection, we also illustrate how to apply our bounds in some examples
and plot the output.

\begin{notation}
\label{not:bar}
Let $\mR$ be a $G$-system and
$m(\mR)=|\mR_1| + \ldots + |\mR_k|$.
We define the integer vector
$$y(\mR)= (|R| \, : \, i \in \{1, \ldots, k\}, \, R \in \mR_i) \in \Z_{\ge 0}^{m(\mR)},$$ 
where we take the same order as in Definition~\ref{def:all}.  Note that all the entries of $y(\mR)$ are positive since recovery sets are nonempty by definition.
\end{notation}

We can now state the main result of this section, which gives an infinite number of half-spaces that contain the allocation polytope, one for each vector $c \in \R^{m(\mR)}$.

\begin{theorem}
\label{thm:clippedsum}
Let~$\mR$ be a $G$-system, $m=m(\mR)$, and let
$c \in \R^m$. Define $y=y(\mR)$
and let $\pi: \{1, \ldots, m\} \to \{1, \ldots, m\}$ be any permutation such that
$$\frac{c_{\pi(1)}}{y_{\pi(1)}} \ge \cdots \ge \frac{c_{\pi(m)}}{y_{\pi(m)}}.$$
Define
$J=\{j \mid y_{\pi(1)} + \ldots + y_{\pi(j)} > n\}$.
If $J= \emptyset$ then let $r=m+1$,
$\sigma=0$, and $\pi(m+1)=0$.
If $J \neq \emptyset$ then let
$r=\min(J)$
and 
$\smash{\sigma = \left(n-\sum_{j=1}^{r-1} y_{\pi(j)}\right)/y_{\pi(r)}}$.
Then for any 
$x \in \mA(\mR)$
we have 
\begin{equation}
\label{eq:clip}
c x^\top \le \left(\sum_{j=1}^{r-1} c_{\pi(j)}\right) +c_{\pi(r)} \, \sigma, ~~
\text{where} ~~ c_{\pi(m+1)}=0.
\end{equation} 
\end{theorem}

Before proving the theorem, we state an immediate consequence for the max-sum capacity of the service rate region.
The result is obtained by taking $c=(1, \ldots, 1)$, allowing us to use a more efficient notation.

\begin{corollary}
\label{prop:max-sum}
Let $\mR$ be any $G$-system with the property that
$\Lambda(\mR) = \Lambda(G)$. 
Let $y=y(\mR)$ and reorder its component nondecreasingly obtaining a vector $\hat{y}$.
Suppose $\hat{y}_1 + \ldots + \hat{y}_m>n$,
where $m=m(\mR)$, and let $r=\min\{j \mid \hat{y}_1 + \ldots + \hat{y}_j >n\}$.
We have
$$\lambda(G) \le r -1 + \frac{n-\sum_{j=1}^{r-1} \hat{y}_j}{\hat{y}_r}.$$
\end{corollary}

We give an example illustrating how to apply Corollary~\ref{prop:max-sum}.

\begin{example} \label{ex:513}
Let~$G$ be as in Example~\ref{ex:otherMDS}, with $n=4$, and $\mR=\mR^{\min}(G)$.
We have $y=y(\mR) = (1,3,1,3,1,3)$.
As in Corollary~\ref{prop:max-sum},
we construct $\hat{y}=(1,1,1,3,3,3)$
and obtain
$\lambda(G) \le 10/3$.
It can be checked that
$\lambda(G) =3$.
\end{example}

\begin{proof}[Proof of Theorem~\ref{thm:clippedsum}]
Let $\mP=\{x \in [0,1]^{m(\mR)} \mid y(\mR) x^\top \le n\}$,
which is a relaxed knapsack polytope.
Let $\beta = \max\{cx^\top \mid x \in \mP\}$.
We have the inclusion 
$\mA(\mR) \subseteq \mP$, hence
$\max\{cx^\top \mid x \in \mA(\mR)\} \le 
\beta$. 
We apply a classical result by Dantzig~\cite{dantzig1957discrete}
(the case where $J=\emptyset$ requires a separate treatment, but we omit it here),
which states that a point $\hat{x} \in \mP$ attaining the maximum $\beta$ is given by:
$$
\hat{x}_j = \left\{ \begin{array}{cl}
      1 & \mbox{if $1 \le j \le r-1$,} \\
      \displaystyle\frac{n-\sum_{j=1}^{r-1} y_{\pi(j)}}{y_{\pi(r)}} & \mbox{if $j=r$,} \\
      0 & \mbox{otherwise.}
\end{array} \right.$$
The result follows by computing $\mu=c \hat{x}^\top$.
\end{proof}

As another corollary of Theorem~\ref{thm:clippedsum}, we obtain a result
that gives an infinite number of half-spaces in which the service rate region~$\Lambda(\mR)$ is contained.
Each half-space is obtained by choosing a different vector $b$ in the statement.

\begin{corollary}
\label{cor:clipSRR}
Let $\mR$ be a $G$-system and let 
$b \in \R^k$.
Let $m=m(\mR)$, $m_0=1$, and $m_i=|\mR_i|$ for all $i \in \{1, \ldots, k\}$.
Define $c \in \R^{m(\mR)}$  by setting
$c_j = b_i$ whenever 
$m_{i-1}+1 \le j \le m_i$.
Construct a permutation $\pi$ and define $r$, $\sigma$, and $\pi(m+1)$ if necessary, as in 
Theorem~\ref{thm:clippedsum}.
Then for all $\lambda \in \Lambda(\mR)$ we have
$$b \lambda^\top \le 
 \left(\sum_{j=1}^{r-1} c_{\pi(j)}\right) +c_{\pi(r)} \, \sigma.$$
\end{corollary}

By specializing the previous result to
vectors $b \in \{0,1\}^k$ one can obtain
upper bounds for partial sums of the form $\sum_{i \in I} \lambda_i$,
where $I \subseteq \{1, \ldots, k\}$
and $\lambda \in \Lambda(\mR)$. 
In particular, one can obtain 
an upper bound for the max-sum capacity $\lambda(G)$.
We conclude this section by illustrating how Corollary \ref{cor:clipSRR} can be applied and the type of results it gives.

\begin{example}
\label{ex:opt}
Consider the service rate region
of Example \ref{ex:different_examples},
depicted in Figure \ref{subfig:2}.
By applying Corollary \ref{cor:clipSRR} for all $b \in \{0,1\}^3$
we obtain the bounding polytope for the service rate region, and we 
depict it in Figure~\ref{fig:opt} as well as with $\Lambda(G).$
\end{example}

\vspace{-0.2cm}
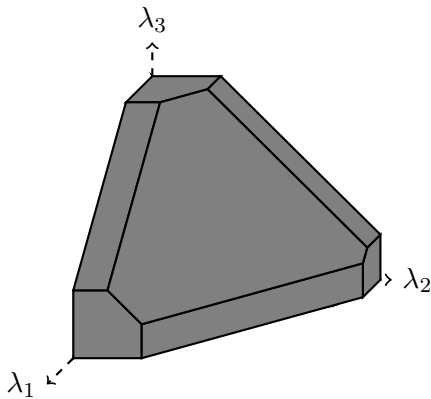
\begin{figure}[hbt!]
    \centering
\begin{tikzpicture}[thick,scale=0.9]

\coordinate (D1) at (0,3,1);
\coordinate (D2) at (1,0,3);
\coordinate (D3) at (1/2,3,1);
\coordinate (D4) at (1/2,1,3);
\coordinate (D5) at (10/3,2/3,1/2);
\coordinate (D6) at (1,1/2,3);
\coordinate (D7) at (1,3,1/2);
\coordinate (D8) at (10/3,1/2,2/3);
\coordinate (D9) at (10/3,0,2/3);

\coordinate (C1) at (0,1,3);
\coordinate (C2) at (10/3,2/3,0);
\coordinate (C3) at (1,3,0);
\coordinate (C4) at (10/3,0,0);

\coordinate (A1) at (0,0,0);
\coordinate (A2) at (0,3,0);
\coordinate (A3) at (0,0,3);
\coordinate (A4) at (3,0,0);
\coordinate (A5) at (3/2,1/2,2);
\coordinate (A6) at (1,1,2);
\coordinate (A7) at (1,0,5/2);
\coordinate (B1) at (3,0,1);
\coordinate (B2) at (1,2,1);

\begin{scope}[thick,dashed,opacity=0.5]
\draw (A1) -- (A2);
\draw (A1) -- (A3);
\draw (A1) -- (A4);

\end{scope}

\draw[fill=blue!50,opacity=0.4] (A2) -- (B2) -- (A6) -- (A3) -- cycle;
\draw[fill=blue!50,opacity=0.4] (A3) -- (A6) -- (A5) -- (A7) -- cycle;
\draw[fill=blue!50,opacity=0.4] (A5) -- (A7) -- (B1) -- cycle;
\draw[fill=blue!50,opacity=0.4] (A5) -- (A6) -- (B2) -- (B1) -- cycle;
\draw[fill=blue!50,opacity=0.4] (B1) -- (A4) -- (A2) -- (B2) -- cycle;
\draw[->, dashed, thick,black] (3,0,0)--(3.5,0,0) node[right]{$\lambda_2$};
\draw[->, dashed, thick,black] (0,0,3)--(0,0,4) node[left]{$\lambda_1$};
\draw[->, dashed, thick,black] (0,3,0)--(0,3.5,0) node[above]{$\lambda_3$};
\draw[thick,black] (A4)--(B1)--(A7)--(A3)--(A2)--(A4);

\draw[fill=gray,opacity=0.4] (A2) -- (D1) -- (D3) -- (D7) --(C3) -- cycle;
\draw[fill=gray,opacity=0.4] (C2) -- (C4) -- (D9) -- (D8) --(D5) -- cycle;
\draw[fill=gray,opacity=0.4] (A3) -- (C1) -- (D4) -- (D6) --(D2) -- cycle;
\draw[fill=gray,opacity=0.4] (D2) -- (D6) -- (D8) -- (D9) -- cycle;
\draw[fill=gray,opacity=0.4] (C1) -- (D1) -- (D3) -- (D4) -- cycle;
\draw[fill=gray,opacity=0.4] (C3) -- (D7) -- (D5) -- (C2) -- cycle;
\draw[fill=gray,opacity=0.4] (D3) -- (D7) -- (D5) -- (D8) -- (D6) -- (D4) -- cycle;
\end{tikzpicture}
    \caption{Service rate region and bounding polytope for Example \ref{ex:opt}.} 
    \label{fig:opt}
\end{figure}

The following example shows that applying Corollary~\ref{cor:clipSRR} not only with 0-1 vectors can give a strictly better bound than only applying it with 0-1 vectors.

\begin{example}
\label{ex:optimization}
Let  $$G = \begin{pmatrix} 1 & 0 & 1 & 1 & 0 & 0 & 1 & 1 \\ 0 & 1 & 0 & 0 & 1 & 1 & 1 & 2 \end{pmatrix} \in \F_3^{2 \times 8}.$$ The service rate region $\Lambda(G)$ is the purple region in Figure~\ref{lastE}. By applying Corollary \ref{cor:clipSRR} for all $b \in \{0,1\}^2$, one gets the light blue region in Figure \ref{lastE}. For a better approximation of the service rate region, also depicted in Figure~\ref{lastE}, we can use Corollary~\ref{cor:clipSRR}. For example, by applying said corollary with $b = (3,2)$ and $b=(3,5)$, in addition to the 0-1 vectors $b \in \{0,1\}^2$,
one gets
the gray region in Figure~\ref{lastE}.
\end{example}

\vspace{-0.6cm}

\begin{figure}[hbt]
\centering

\begin{tikzpicture}[thick,scale=0.5]
\coordinate (A1) at (0,0);
\coordinate (A2) at (5,1);
\coordinate (A3) at (5,0);
\coordinate (A4) at (4,3);
\coordinate (A5) at (3,4);
\coordinate (A6) at (0,5);
\coordinate (A7) at (1,5);

\coordinate (B1) at (11/2,0);
\coordinate (B2) at (0,11/2);
\coordinate (B3) at (11/2,3/2);
\coordinate (B4) at (3/2,11/2);

\coordinate (C1) at (11/2,3/4);
\coordinate (C2) at (1/2,11/2);

\draw[->, thick,black] (0,0)--(6,0) node[right]{$\lambda_1$};
\draw[->, thick,black] (0,0)--(0,6) node[above]{$\lambda_2$};

\draw[thick,blue] (A3)--(A2)--(A4)--(A5)--(A7)--(A6);
\draw[thick,cyan] (B1)--(B3) -- (B4) --(B2);
\draw[thick,gray] (B1)--(C1)--(A4)--(A5)--(C2)--(B2);

\draw[fill=gray!50] (A3)--(A2)-- (A4) -- (C1)-- (B1)-- cycle;

\draw[fill=gray!50] (A6)--(A7)-- (A5) -- (C2)-- (B2)-- cycle;

\draw[fill=blue!50,opacity=0.4] (A3)--(A2)--(A4)--(A5)--(A7)--(A6) -- (A1) -- cycle;

\draw[fill=cyan!50] (C2)--(B4) -- (A5) -- cycle;

\draw[fill=cyan!50] (C1)--(B3) -- (A4) -- cycle;

\end{tikzpicture}
    \caption{The service rate region and the bounding polytopes for
    Example \ref{ex:optimization}.
    \label{lastE}}
\end{figure}
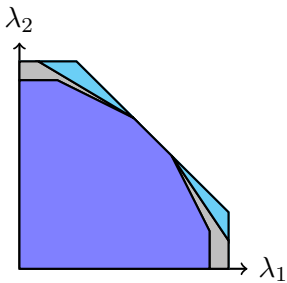

\section{Systematic MDS Codes}
\label{sec:6}

This section is entirely devoted to the service rate region 
$\Lambda(G)$, when $G$ is an MDS matrix;
see Appendix~\ref{sec:appendix2} for the definition of MDS matrix.
We focus on the volumes of these service rate regions for $k \in \{2,3\}$,
and on their max-sum capacities.

Recall that in the case where $G$ is a systematic MDS matrix and $n \ge 2k$, the service rate region
$\Lambda(G)$ is known and given by the set
\begin{equation} \label{eq:mds}
    \left\{(\lambda_1,\ldots,\lambda_k) \in \R_{\ge 0}^k \ \Bigl\lvert \ \sum_{i=1}^k \Bigl( \min\{\lambda_i,1\} + k \cdot \max\{0,\lambda_i-1\}  \Bigr) \le n \right\};
\end{equation}
see Remark~\ref{rem:sharp}.
The description in~\eqref{eq:mds} is however inconvenient for computing the volume of $\Lambda(G)$, which is one of the goals of this section. Therefore, our first move is deriving a more convenient description.

\begin{notation}
\label{not:pi}
Given a vector $\lambda \in \R^k$, let $\chi(\lambda) \in \Z^k$ be the vector with 
$\chi(\lambda)_i=1$ if $\lambda_i < 1$ and
$\chi(\lambda)_i=k$ if $\lambda_i \ge 1$. Moreover, we let
$$\lambda_{<1} = \{i \in \{1,\ldots,k\} \mid \lambda_i <1\}.$$
\end{notation}

The following lemma gives a different representation of the service rate region of a systematic MDS matrix $G$ with $n \ge 2k$. In its statement, we use
Notation~\ref{not:pi}.

\begin{lemma}
\label{lem:ddb}
Suppose $G$ is a systematic MDS matrix and $n \ge 2k$. We have
\begin{equation}
\label{eq:nkp}
\Lambda(G)= \left\{\lambda \in \R_{\ge 0}^k \ \Bigl\lvert \  \chi(\lambda) \, \lambda^\top \le n + (k-1)(k-|\lambda_{<1}| ) \right\}.
\end{equation}
\end{lemma}

\begin{proof}
Let $\lambda \in \Lambda(G)$. 
By~\eqref{eq:mds}, we have $\sum_{i=1}^k \Bigl( \min\{\lambda_i,1\} + k \cdot \max\{0,\lambda_i-1\}  \Bigr) \le n.$ That is,  
\begin{equation*}
\sum_{i \in \lambda_{<1}}\lambda_i + (k-|\lambda_{<1}|) + \left(\sum_{i \in \{1,\ldots,k\} \setminus \lambda_{<1}} k\lambda_i\right) - k(k-|\lambda_{<1}|)\le n.
\end{equation*}
The latter inequality can be rewritten as
$\chi(\lambda) \cdot \lambda^\top \le n + (k-1)(k- |\lambda_{<1}| ).$ 
This shows the inclusion $\subseteq$ in \eqref{eq:nkp}. The other inclusion follows by reversing all the passages, and we omit the details.
\end{proof}

We can now compute the volume
of the service rate region 
of an MDS matrix for $k \in \{2,3\}$ and 
$n \ge 2k$.
We start with the case $k=2$.

\begin{theorem}
\label{thm:vol2}
Let~$G \in \F_q^{2 \times n}$ be a systematic MDS matrix. Suppose $n\ge 4$. Then we have $\smash{\vol(\Lambda(G))=\frac{n^2+4n}{8}}$.
\end{theorem}

\begin{proof}
By Lemma~\ref{lem:ddb}, 
$\Lambda(G)$
is defined by the following
five equations:
$$\lambda_1 + \lambda_2 \le \frac{n+2}{2}, \qquad 
2\lambda_1 + \lambda_2 \le n+1, \qquad
\lambda_1 + 2\lambda_2 \le n+1, \qquad
\lambda_1 \ge 0, \qquad
\lambda_2 \ge 0.$$
It is not hard to check that
the vertices of $\Lambda(G)$ are
the points
$$(0,0),\qquad \left(1,\frac{n}{2}\right),\qquad \left(\frac{n}{2},1\right),\qquad \left(\frac{n+1}{2},0\right), \qquad 
\left(0,\frac{n+1}{2}\right).$$
The volume (i.e., the area) can now be computed using elementary methods.
\end{proof}

We can compare Theorem~\ref{thm:vol2} with a replication system generated by a matrix with the same parameters. 

\begin{proposition} \label{prop:vol_repl}
Suppose that~$G \in \F_q^{k \times n}$ is a replication matrix. We have $$n-k+1 \le \vol(\Lambda(G)) \le \floor{\left(\frac{n}{k}\right)^k}.$$
The lower bound can  be attained by some $G$ and 
the upper bound can be attained 
by some $G$ if $k=2$. 
\end{proposition}
\begin{proof}
Let $j_i$ denote the number of columns of $G$ multiples of the standard basis vector~$e_i$, for $i \in \{1,\ldots,k\}.$ Each $j_i$ is a positive integer since $G$ is full rank. By Proposition~\ref{prop: volser}, the volume of $\Lambda(G)$ is $\prod_{i=1}^kj_i.$ We get the desired upper bound by the arithmetic vs.\ geometric mean inequality.
The lower can be attained by taking $j_1=n-k+1$ and $j_i=1$ for $i \in \{2,\ldots,k\}$. When
$k=2$, the upper bound can be attained by taking $j_1 \in \{\lfloor n/2 \rfloor,\lceil n/2 \rceil\}$ and $j_2 = n - j_1$.
\end{proof}

Note that Proposition~\ref{prop:vol_repl} shows that one can find a replication matrix $G \in \F_q^{2 \times n}$
for $n \ge 5$ whose service rate region's volume is strictly larger than the volume of an MDS matrix of the same size.

We now turn to the case $k=3$ and $n \ge 6$,
computing the volume of the service rate region corresponding to an MDS matrix $\smash{G \in \F_q^{3 \times n}}$. The computation is more involved than in the 2-dimensional case.

\begin{theorem}
\label{thm:volume3D}
Let $\smash{G \in \F_q^{3 \times n}}$ be a systematic MDS matrix and suppose $n\ge 6$.
We have 
$${\vol(\Lambda(G))= \frac{n^3+18n^2+54n-18}{162}}.$$ 
\end{theorem}
\begin{proof}
Define the function $f(z)=\min\{z,1\} + 3\max\{0,z-1\}$. 
Using Lemma~\ref{lem:ddb} it can be seen that
$\Lambda(G)$ is the set of 3-tuples $(\lambda_1,\lambda_2,z)$
that satisfy the inequalities
$$\begin{cases}      
      \lambda_1 + \lambda_2 \le n - f(z) \qquad \qquad (*) \\
      3\lambda_1 + \lambda_2 \le n - f(z) + 2  \\
      \lambda_1 + 3\lambda_2 \le n - f(z) + 2 \\
      3\lambda_1 + 3\lambda_2 \le n - f(z) + 4 \label{a2} \quad (**) \\
      \lambda_1, \lambda_2, z \ge 0 \\
\end{cases}$$
Observe moreover that
the maximum value $z$ can take is $(n+2)/3$. This value can be attained by taking~$\lambda_1 = \lambda_2 = 0$ in the above system. 
We have 
$$f(z)=
\begin{cases}      
      z  & \text{if} \ \ 0 \le z \le 1, \\
      3z - 2  & \text{if} \ \ 1 \le z.
\end{cases}$$
It can be checked that $(**)$ is more restrictive than $(*)$
for $z \le n/3$, while $(*)$ 
is more restrictive than $(**)$
otherwise. Moreover, when $z \ge (n+1)/3$, all inequalities except for  $(*)$ and the non-negativity of $\lambda_1$, $\lambda_2$ and $z$ can be disregarded. This tells us the shape of the ``slices'' of $\Lambda(G)$ for a given value of $z$. We summarize this discussion in the following table and Figure~\ref{ffff}.

\begin{table}[!hbt]
\begin{center}
\caption{}
 \renewcommand{\arraystretch}{1.4}
  \begin{tabular}{|c|c|c|c|c|c|} 
 \hline
  $z$ & Figure & $y$ & $x$ & $\alpha$ & $\beta$  \\ [0.5ex] 
 \hline\hline
  $0 \le z \le 1$ & \ref{fig:2dim1} & $(0,\frac{n-z+2}{3})$ & $(0,\frac{n-z+2}{3})$ & $(1,\frac{n-z+1}{3})$ & $(\frac{n-z+1}{3},1)$  \\
 \hline
 $1 \le z \le \frac{n}{3}$ &  \ref{fig:2dim1} & $(0,\frac{n-3z+4}{3})$ & $(0,\frac{n-3z+4}{3})$ & $(1,\frac{n-3z+3}{3})$ & $(\frac{n-3z+3}{3},1)$  \\
\hline
 $\frac{n}{3} \le z \le \frac{n+1}{3}$ &  \ref{fig:2dim1} & $(0,\frac{n-3z+4}{3})$ & $(0,\frac{n-3z+4}{3})$ & $(n-3z+1,1)$ & $(1,n-3z+1)$ \\
 \hline
  $\frac{n+1}{3} \le z \le \frac{n+2}{3}$ &  \ref{fig:2dim2} & $(0,n-3z+2)$ & $(n-3z+2,0)$ & --- & ---  \\
 \hline
\end{tabular}
\end{center}
\label{tt}
\end{table}

\begin{figure}[hbt!]
    \centering
\subcaptionbox{The typical service rate region of a systematic MDS matrix $G \in \F_q^{3 \times n}$
     for $n \ge 6$.}
[.3\linewidth]{\begin{tikzpicture}[thick,scale=0.8]
 \coordinate (A1) at (0,0,0);
 \coordinate (A2) at (0,0,8/3);
 \coordinate (A3) at (8/3,0,0);
 \coordinate (A4) at (0,8/3,0);
 \coordinate (A5) at (1,0,7/3);
 \coordinate (A6) at (0,1,7/3);
 \coordinate (A7) at (7/3,0,1);
 \coordinate (B1) at (0,7/3,1);
 \coordinate (B2) at (7/3,1,0);
 \coordinate (B3) at (1,7/3,0);
 \coordinate (B4) at (1,1,2);
 \coordinate (B5) at (1,2,1);
 \coordinate (B6) at (2,1,1);

 \begin{scope}[thick,dashed,,opacity=0.6]
 \draw (A1) -- (A2);
 \draw (A1) -- (A3);
 \draw (A1) -- (A4);

 \end{scope}
 \draw[fill=blue!50,opacity=0.4] (A2) -- (A5) -- (B4) -- (A6) -- cycle;
 \draw[fill=blue!50,opacity=0.4] (A5) -- (B4) -- (B6) -- (A7) -- cycle;
 \draw[fill=blue!50,opacity=0.4] (A7) -- (B6) -- (B2) -- (A3) -- cycle;
 \draw[fill=blue!50,opacity=0.4] (B2) -- (B6) -- (B5) -- (B3) -- cycle;
 \draw[fill=blue!50,opacity=0.4] (B3) -- (B5) -- (B1) -- (A4) -- cycle;
 \draw[fill=blue!50,opacity=0.4] (B5) -- (B1) -- (A6) -- (B4) -- cycle;
 \draw[fill=blue!50,opacity=0.4] (B6) -- (B5) -- (B4) -- cycle;
 \end{tikzpicture}}
 \hspace{.03\textwidth}
\subcaptionbox{The slice of the service rate region of a systematic MDS matrix $\smash{G \in \F_q^{3 \times n}}$, $n \ge 6$,
    for $z \le (n+1)/3$.\label{fig:2dim1}}
[.3\linewidth]{\begin{tikzpicture}[thick,scale=0.9]
\coordinate (A1) at (0,0);
\coordinate (A2) at (0,5/2);
\coordinate (A3) at (5/2,0);
\coordinate (A4) at (1,2);
\coordinate (A5) at (2,1);
\node at (A2) [left = 1mm of A2] {$y$};
\node at (A3) [below = 1mm of A3] {$x$};
\node at (A4) [right = 1mm of A4] {$\alpha$};
\node at (A5) [right = 1mm of A5] {$\beta$};

\draw[fill=blue!50,opacity=0.4] (A1) -- (A2) -- (A4) -- (A5) -- (A3) -- cycle;
\draw[->, thick,black] (0,0)--(3,0) node[right]{$\lambda_1$};
\draw[->, thick,black] (0,0)--(0,3) node[above]{$\lambda_2$};
\draw[thick,black] (0,2.5)--(1,2)--(2,1)--(2.5,0);
\end{tikzpicture}}
 \hspace{.03\textwidth}
\subcaptionbox{The slice of the service rate region of a systematic MDS matrix $\smash{G \in \F_q^{3 \times n}}$, $n \ge 6$, for  $(n+1)/3 \le z \le (n+2)/3$.\label{fig:2dim2}}
[.3\linewidth]{\begin{tikzpicture}[thick,scale=0.9]
\coordinate (A1) at (0,0);
\coordinate (A2) at (0,5/2);
\coordinate (A3) at (5/2,0);

\node at (A2) [left = 1mm of A2] {$y$};
\node at (A3) [below = 1mm of A3] {$x$};

\draw[fill=blue!50,opacity=0.4] (A1) -- (A2) -- (A3) -- cycle;
\draw[->, thick,black] (0,0)--(3,0) node[right]{$\lambda_1$};
\draw[->, thick,black] (0,0)--(0,3) node[above]{$\lambda_2$};
\draw[thick,black] (0,2.5)--(2.5,0);
\end{tikzpicture}}
\caption{The service rate region and its slices
for the proof of Theorem~\ref{thm:volume3D}. \label{ffff}}
\end{figure}
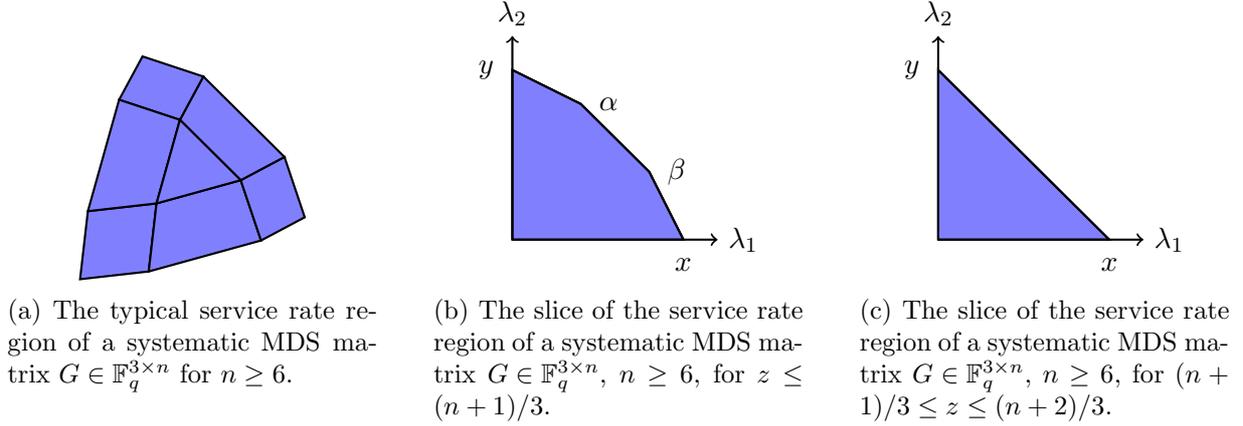

The areas of the slices can be easily computed and therefore
the volume of $\Lambda(G)$ can be computed by integration over $z$.
This approach is mathematically justified, for example, 
by~\cite[Theorem 2.7]{apostol1991calculus}. 
The desired formula follows from
\begin{align*}
\vol(\Lambda(G)) &= \int_{0}^{1} \Big(\frac{1}{18}z^2 - \frac{n+4}{9}z + \frac{n^2+8n+4}{18}\Big) \,dz + \int_{1}^{n/3} \Big(\frac{1}{2}z^2 - \frac{n+6}{3}z + \frac{n^2+12n+24}{18}\Big) \,dz \\
& \hspace{3cm} + \int_{n/3}^{(n+1)/3} \Big(-\frac{3}{2}z^2 + (n-2)z + \frac{n^2-4n-8}{6}\Big) \,dz \\ & \hspace{5cm} + \int_{(n+1)/3}^{(n+2)/3} \Big(\frac{9}{2}z^2 - (3n + 6)z + \frac{n^2+4n+4}{2}\Big) \,dz
\end{align*}
and tedious but straightforward computations.
\end{proof}

Out of curiosity, we point out that
Theorem~\ref{thm:volume3D} can also be derived by the well-known \textit{triangulation method} for computing the volume of a polytope using the volume of simplices; see~\cite{cohen1979two}.
For the polytope of 
Theorem~\ref{thm:volume3D} the formalization of this approach is rather involved, which is why we proceeded by integration.

We also notice that a general lower bound for $\vol(\Lambda(G))$ where~$G$ is an MDS matrix can be obtained by Proposition \ref{prop:indep}. Any $k$ columns can be used to recover any data object. Thus, $(n/k)e_i \in \Lambda(G)$, which implies that the simplex with these vertices is contained in the service rate region. Therefore, 
$\vol(\Lambda(G)) \ge (n/k)^k/k!.$

In the second part of this section, we investigate
other parameters of the service rate regions of systematic MDS matrices. We first observe that Corollary~\ref{prop:max-sum} implies the following.

\begin{corollary}
\label{cor:maxMDS}
Let~$G \in \F_q^{k \times n}$ be a systematic MDS matrix. We have 
$\lambda(G) \le k + \frac{n-k}{k}.$ 
\end{corollary}
\begin{proof}
Following the notation of 
Corollary~\ref{prop:max-sum},
we have
$m=k\left(\binom{n-1}{k}+1\right)$ and 
$\hat{y}=(v_1,v_2)$, where
$v_1 = (1,\ldots,1) \in \R^k$ and 
$v_2 = (k,\ldots,k) \in \R^{m-k}$.
Moreover,
$$r = 
\min\{j \mid \hat{y}_1 + \ldots + \hat{y}_j > n,\ 1 \le j \le m\}
=k + \floor{(n-k)/k} +1.$$
We then obtain the desired result by applying 
Corollary~\ref{prop:max-sum}:
\begin{align*}
\lambda(G) &\le k + \floor{(n-k)/k} + \frac{n-\sum_{i=1}^{r-1} \hat{y}_i}{\hat{y}_r} = k + \frac{n-k}{k}. \end{align*}
\end{proof}

We will now prove that systematic MDS matrices achieve the bound of Corollary~\ref{cor:maxMDS} with equality in the case~$n \ge 2k$. We start by introducing
some objects 
that we will need 
to prove the result.

\begin{notation}
\label{not:emina}
Let $a,b \in \Z$ such that $2 \le a < b$.  Assume $q > b$ and let $\alpha$ be a primitive element of $\F_q$. Define the matrix
\[ 
G^{a,b}_\alpha=\begin{pmatrix}
    1 & 1 & 1& \cdots & 1 \\
     1 & \alpha & \alpha^2 & \cdots & \alpha^{b-1}\\
   %  1 & \alpha^2 & \alpha^4 & \cdots & \alpha^{2(b-1)}\\
      \vdots & \vdots & \vdots & \ddots & \vdots \\
   1 & \alpha^{(a-1)} & \alpha^{2(a-1)} & \cdots & \alpha^{(b-1)(a-1)} 
  \end{pmatrix} \in \F_q^{a \times b}.
\]
\end{notation}

Note that having $q>b$
is necessary and sufficient 
for the columns of $G^{a,b}_\alpha$
to be pairwise distinct. Sufficiency can be seen
by considering the second row of
$G^{a,b}_\alpha$ and the fact that the multiplicative order of $\alpha$ is $q-1$.
Necessity follows from the fact that if $q \le b$ then at least one of the columns of $G^{a,b}_\alpha$ indexed by
$\{2, \ldots, b\}$ is equal to the first column.
Moreover, the matrix $G^{a,b}_\alpha$ is a generator matrix of a \textit{Reed-Solomon code} \cite{reed1960polynomial}, which is a type of MDS code; see \cite{macwilliams1977theory}. 

\begin{lemma}
\label{lem:RNC}
Following Notation~\ref{not:emina}, let
$\mR=\mR^{\min}(G^{a,b}_\alpha)$. The following hold.
\begin{enumerate}
    \item Let $i \in \{1,\ldots,a\}$ and
    $R \subseteq \{1, \ldots, b\}$.
    Then $R \in \mR_i$ if and only if $|R|=a$.
    
    \item Let $\nu \in \{1, \ldots, b\}$.
   We have
    $$|\{R \in \mR_i \mid i \in \{1,\ldots,a\}, \, \nu \in R \}| = \binom{b-1}{a-1}.$$
\end{enumerate}
\end{lemma}

\begin{proof}
We first observe that the second part of the lemma follows from the first and the fact that
$$\{S \subseteq \{1,\ldots,b\} \mid |S| = a,\ \nu \in S\} = \binom{b-1}{a-1}~ \text {for all $1 \le \nu \le b$.} $$

To prove the first part,
let 
$G = G^{a,b}_\alpha$.
Suppose that $R \in \mR_i$ and let us prove $|R|=a$. We first show that $|R| \le a$. Assume towards a contradiction that $|R| >a.$ By Proposition \ref{prop:indep}, there exists $R' \subseteq R$ such that $|R|=a$ and $R' \in \mR_i$, which
contradicts the fact that $R$ is $i$-minimal. We now show that
$|R| \ge a$. Towards a contradiction, assume $|R| = c < a.$ Because of the structure of $G$, we can assume without loss of generality
$i=a$ and $R=\{1, \ldots, c\}$.
We will prove that $R \notin \mR_a$, which is a contradiction.
That is equivalent to showing that $e_a \notin \langle G^\nu \mid \nu \in \{1,\ldots,c\}\rangle$, which can be seen from the fact that the matrices
\begin{equation*}
\begin{pmatrix}
    1 & 1 & 1& \cdots & 1 \\
     1 & \alpha & \alpha^2 & \cdots & \alpha^{c-1}\\
   %  1 & \alpha^2 & \alpha^4 & \cdots & \alpha^{2(c-1)}\\
      \vdots & \vdots & \vdots & \ddots & \vdots \\
   1 & \alpha^{(a-1)} & \alpha^{2(a-1)} & \cdots & \alpha^{(c-1)(a-1)} 
  \end{pmatrix} \mbox{\ and \ } \begin{pmatrix}
    1 & 1 & 1& \cdots & 1 & 0 \\
     1 & \alpha & \alpha^2 & \cdots & \alpha^{c-1} & 0\\
   %  1 & \alpha^2 & \alpha^4 & \cdots & \alpha^{2(c-1)}\\
      \vdots & \vdots & \vdots & \ddots & \vdots & \vdots\\
   1 & \alpha^{(a-1)} & \alpha^{2(a-1)} & \cdots & \alpha^{(c-1)(a-1)} & 1 
  \end{pmatrix}
\end{equation*}
have different ranks. 
For the other direction,
assume $|R|=a$.
 Since $G$ is an MDS matrix, we have $R \in \mR_i^\all(G)$ by Proposition~\ref{prop:indep}. 
To see that 
$R \in \mR_i$ it is enough to show that all elements of $\mR_i$ have cardinality $a$, which we proved already in the first part of the proof.
\end{proof}

Note that the second property of the previous lemma states that
each column index of $G^{a,b}_\alpha$ participates in $\binom{b-1}{a-1}$ recovery sets of the system $\mR(G^{a,b}_\alpha)$.

\begin{theorem}
\label{thm:interesting}
Let $(\lambda_1, \ldots, \lambda_a) \in \R_{\ge 0}^a$.
Following Notation~\ref{not:emina}, if
$\lambda_1+\dots+\lambda_a\le b/a$, then ~$\smash{(\lambda_1,\dots,\lambda_a) \in \Lambda(G^{a,b}_\alpha)}$.
\end{theorem}

\begin{proof}
Let $\mR=\mR^{\min}(G^{a,b}_\alpha)$. For all $i \in \{1, \ldots, a\}$ and $R \in \mR_i$, let
$$\lambda_{i,R} = \frac{\lambda_i}{\binom{b}{a}} = \lambda_i \,  \frac {a}{b} \,  \frac{1}{\binom{b-1}{a-1}}.$$
We now show that the constraints in Definition \ref{def:system} hold.
Constraint \eqref{c3} holds by definition.
Constraint \eqref{c1} is satisfied because
$$\sum_{R \in \mR_i} \lambda_{i,R} = \lambda_i \,  \frac {a}{b} \,  \frac{1}{\binom{b-1}{a-1}} \, |\mR_i| =  \lambda_i \,  \frac {a}{b} \,  \frac{1}{\binom{b-1}{a-1}} \, \binom{b}{a} = \lambda_i$$ for all $i \in \{1,\ldots,a\}$, where
the fact that $|\mR_i|=\binom{b}{a}$ follows from the first part of 
Lemma~\ref{lem:RNC}.
By the second part of Lemma \ref{lem:RNC}, constraint \eqref{c2} reads as
\begin{equation}
\label{capacity_node}
\binom{b-1}{a-1} \sum_{i=1}^a \lambda_i\, \frac {a}{b}\, \frac{1}{\binom{b-1}{a-1}} \le 1,
\end{equation}
which holds by the theorem's assumption  $\lambda_1 +\ldots + \lambda_a \le b/a$, since 
\begin{equation*}
    \binom{b-1}{a-1} \sum_{i=1}^a \lambda_i\, \frac {a}{b}\, \frac{1}{\binom{b-1}{a-1}} = \frac{a}{b} (\lambda_1 + \ldots + \lambda_a) \le \frac{a}{b}\, \frac{b}{a} = 1.
\end{equation*}
\end{proof}

We can now show that systematic MDS matrices with $n \ge 2k$ achieve 
the bound of Corollary~\ref{cor:maxMDS} (cf.~\cite{service:aktas2021jkks}).
\begin{theorem}
Suppose $n \ge 2k$.
If $G \in \F_q^{k \times n}$ is a systematic MDS matrix, then
$\smash{\lambda(G) = k+ \frac{n-k}{k}}$.
\end{theorem}
\begin{proof}
All systematic MDS matrices with $n \ge 2k$ have the same service rate region; see Remark~\ref{rem:sharp}.
Therefore it suffices to prove the result
for $\smash{G = \bigl[\textnormal{Id}_k \mid G^{k,n-k}_\alpha \bigr] \in \F_q^{k \times n}}$, where $\textnormal{Id}_{k}$ is the $k\times k$ identity matrix over $\F_q$ and $G^{k,n-k}_\alpha$ is as in Notation~\ref{not:emina}.

The fact $\lambda_1+\dots+\lambda_k \le k + (n-k)/k$ follows from Corollary~\ref{cor:maxMDS}. Let 
$\smash{\lambda_{i,\{i\}}}= 1$ for all $i \in \{1,\ldots,k\}$, and observe that
$((n-k)/k,0,\ldots,0) \in \Lambda(G^{k,n-k}_\alpha)$ by taking $a= k$ and $b = n-k$ in Theorem~\ref{thm:interesting}. Then $((n-k)/k + 1,1,\ldots,1) \in \Lambda(G)$
by the definition of $G$.
\end{proof}

We conclude this section by noting that
Corollary~\ref{cor:maxMDS} is not necessarily met with equality when $n<2k$, or if $G$ is not systematic. For the case where
$n<2k$, see for instance Example~\ref{ex:513}.
For the case where $n \ge 2k$ and $G$ is a non-systematic MDS matrix, consider
$$G= \begin{pmatrix}
    2 & 1 & 3 &4 \\ 1 & 2 & 3 & 5
\end{pmatrix} \in \F_7^{2 \times 4}.$$
Then $\lambda(G)=2<3$.

\section*{Acknowledgements}
The authors would like to thank Laura Sanità for fruitful discussions on combinatorial optimization and down-monotone polytopes, and the anonymous referees for their comments and suggestions.

\bibliographystyle{abbrv}
\bibliography{serviceL}

\appendix

\section{Polytopes}\label{sec:appendix1}
In this appendix, we collect some background material 
about polytopes and their properties.
More details can be found in standard references; see e.g.~\cite{ziegler2012lectures,korte2011combinatorial,schrijver1998theory}.

We start by recalling that a \textbf{polyhedron} is a set of the form $\mP=\{x \in \R^m \mid A  x^\top \le b^\top\}$, where $A \in \R^{\ell \times m}$, $\ell,m \ge 1$, $b \in \R^\ell$, and $\le$ is applied component-wise. Such set $\mP$ is called a \textbf{polytope} if it is bounded. A fundamental result on polyhedra states that every polytope is the convex hull of a finite set of points. For a (possibly infinite) set $S \subseteq \R^m$, we let $\conv(S)$ denote its convex hull, where $\conv(\emptyset)=\emptyset$.

\begin{theorem}[see e.g.~\cite{grunbaum1967convex}]\label{thm:fundamental_poly}
Let $\mP \subseteq \R^m$ be a polytope. Then a finite set $S \subseteq \R^m$ exists, such as $\mP=\conv(S)$.
\end{theorem}

A \textbf{vertex} of a polytope $\mP \subseteq \R^m$ is an element $v \in \mP$ with $v \notin \conv(\mP \setminus \{v\})$.
The set of vertices of $\mP$ is denoted by $\vert(\mP)$. Note that if $\mP=\conv(S)$ is a poltyope 
then $\vert(\mP) \subseteq S$.
Thus $\mP=\conv(\vert(\mP))$. Moreover, a nonempty polytope has at least one vertex.

We recall the following crucial property of vertices.

\begin{proposition} \label{prop:inv}
Let $\mP = \{x \in \R^m \mid A  x^\top \le b^\top\}$ be a polytope, with $A \in \R^{\ell \times m}$. Let $v$ be a vertex of $\mP$. Then there exists
$I \subseteq \{1, \ldots, \ell\}$ such that $\rank(A[I])=m$ and $\{v\}=\{x \in \R^m \mid A[I]x^\top = b[I]^\top\}$, where $A[I]$ and $b[I]$ are obtained from $A$ and $b$ by deleting the rows and components (respectively) not indexed by $I$. 
\end{proposition}

The previous result remarkably shows that rational polytopes have \textbf{rational} vertices (i.e., with rational entries). Recall that
a polyhedron of the form $\{x \in \R^m \mid Ax\le b^\top\}$ with $A \in \Q^{\ell \times m}$ and $b \in \Q^\ell$ is called \textbf{rational}. Then Proposition~\ref{prop:inv} combined with Gaussian elimination leads to the following result.

\begin{corollary}\label{cor:rational_vert}
A rational polytope has rational vertices.
\end{corollary}

We conclude this appendix by 
recalling 
that
a polytope $\mP \subseteq \R^m$ is \textbf{down-monotone} if 
$x \ge 0$ for all $x \in \mP$ and for all $y \in \R^m$ and $x \in \mP$ with $0 \le y \le x$ we have $y \in \mP$. All polytopes in this paper are down-monotone.

\section{Error-Correcting Codes}
\label{sec:appendix2}

\begin{definition}
An $[n,k]_q$  (\textbf{error-correcting}) \textbf{code} is a $k$-dimensional $\F_q$-linear subspace $\mC \le \F_q^n$. We call $n$ the \textbf{length} of $\mC$.
A matrix $G \in \F_q^{k \times n}$ whose rows span $\mC$ is called a \textbf{generator matrix} for $\mC$.
The 
$[n,n-k]_q$ code
$\mC^\perp = \{x \in \F_q^n \mid xy^\top =0 \mbox{ for all $y \in \mC$}\} \le \F_q^n$
is the \textbf{dual} of $\mC$.
\end{definition}

The error correction capability of $\mC$ is measured by a fundamental parameter defined as follows.

\begin{definition}
The \textbf{Hamming weight} of a vector $x \in \F_q^n$ is the integer
$\wH(x)=|\{i \mid x_i \neq 0\}|$. The \textbf{minimum} (\textbf{Hamming}) \textbf{distance} of a code $\mC \le \F_q^n$~is 
$\dH(\mC)=\min\left\{\wH(x) \mid x \in \mC, \, x \neq 0\right\}$.
\end{definition}

This paper mainly focuses on $[n,k]_q$ codes with $k+d-1=n$. Such codes are called \textbf{MDS} (Maximum Distance Separable). A full rank matrix that generates an MDS code is called an \textbf{MDS matrix}. These matrices are known to exist only over sufficiently large finite fields ($q \ge n-1$ suffices). Determining for which field sizes MDS matrices exist has been an open problem since 1955; see~\cite{segre1955curve}. We conclude with the following handy characterization of MDS matrices.
The proof can be found in~\cite[page 318]{macwilliams1977theory}. 
\begin{proposition}
\label{prop:indep}
Let $G \in \F_q^{k \times n}$ be a matrix. Then $G$ is an MDS matrix if and only if every $k$ columns of $G$ are $\F_q$-linearly independent.
\end{proposition}

\end{document}